\documentclass[11pt,a4paper]{article}

\usepackage{fullpage}
\usepackage{amsmath,amsfonts,amsthm,amssymb}
\usepackage{verbatim} 

\usepackage{esint}

\numberwithin{equation}{section}

\newtheorem{theorem}{Theorem}
\newtheorem{lemma}{Lemma}[section]
\newtheorem*{lemma*}{Lemma}
\newtheorem*{prop*}{Proposition}
\newtheorem*{remark*}{Remark}
\newtheorem{remark}[lemma]{Remark}
\newtheorem{cor}[lemma]{Corollary}
\newtheorem*{cor*}{Corollary}

\providecommand{\ud}{\, \mathrm{d}}
\providecommand{\dxy}{\ud x \ud y}
\providecommand{\dx}{\ud x}
\providecommand{\dy}{\ud y}
\providecommand{\R}{\mathbb{R}}
\providecommand{\oh}{\frac{1}{2}}

\providecommand{\cw}{c_w}
\newcommand{\C}{(C_1 \pi^2)}

\renewcommand{\b}{\bar b}
\renewcommand{\l}{l_0}
\renewcommand{\omega}{w}
\renewcommand{\ln}{\log}

\begin{document}

\title{The coarsening of folds in hanging drapes}

\author{Peter Bella\thanks{Max Planck Institute for Mathematics
in the Sciences, Leipzig, Germany. Email: bella@mis.mpg.de. This work was begun while PB was a
PhD student at the Courant Institute of Mathematical Sciences. Support from NSF grant
DMS-0807347 is gratefully acknowledged.} \ and Robert V. Kohn\thanks{Courant Institute of
Mathematical Sciences, New York University, New York, NY. Email: kohn@cims.nyu.edu. Support is
gratefully acknowledged from NSF grants DMS-0807347, DMS-1311833 and OISE-0967140.} }

\maketitle

\begin{abstract}
We consider the elastic energy of a hanging drape -- a thin elastic sheet, pulled down by the
force of gravity, with fine-scale folding at the top that achieves approximately uniform
confinement. This example of energy-driven pattern formation in a thin elastic sheet is of
particular interest because the length scale of folding varies with height. We focus on how
the minimum elastic energy depends on the physical parameters. As the sheet thickness
vanishes, the limiting energy is due to the gravitational force and is relatively easy to
understand. Our main accomplishment is to identify the ``scaling law'' of the correction due
to positive thickness. We do this by (i) proving an upper bound, by considering the energies
of several constructions and taking the best; (ii) proving an ansatz-free lower bound, which
agrees with the upper bound up to a parameter-independent prefactor. The coarsening of folds
in hanging drapes has also been considered in the recent physics literature, using a
self-similar construction whose basic cell has been called a ``wrinklon.'' Our results
complement and extend that work, by showing that self-similar coarsening achieves the optimal
scaling law in a certain parameter regime, and by showing that other constructions (involving
lateral spreading of the sheet) do better in other regions of parameter space. Our analysis
uses a geometrically linear F\"{o}ppl-von K\'{a}rm\'{a}n model for the elastic energy, and is
restricted to the case when Poisson's ratio is zero.
\end{abstract}

\section{Introduction}

We consider a hanging drape -- a thin elastic sheet, pulled down by the force of gravity, with
fine-scale folding at the top that achieves approximately uniform confinement. We are
interested in how the shape of the sheet varies with height. Since bending costs elastic
energy, one expects to see less bending far from the top. This effect can be achieved by two
rather different mechanisms: coarsening of the folds, or spreading of the sides. Our analysis
includes a study of these two mechanisms and how they interact.

Our viewpoint is variational: we focus on how the minimum elastic energy depends on the
physical parameters. Our elastic energy functional has three terms: the {\it membrane energy}
(which penalizes stretching), the {\it bending energy} (which penalizes bending), and a {\it
loading} term (representing the effect of gravity). The membrane energy is nonconvex, while
the bending energy acts as a regularizing singular perturbation. To capture the essential
behavior in the simplest possible setting, we use a F\"{o}ppl-von K\'{a}rm\'{a}n
model with Poisson's ratio $0$ for the membrane energy. The limiting behavior as the sheet
thickness $h \rightarrow 0$ is then quite easy to understand: the sheet hangs straight down,
under tension due to the force of gravity. The energy of this deformation is easy to find. Our
main accomplishment is to assess (at the level of scaling) the leading-order correction
due to positive $h$. We do this by (i) proving an upper bound on the minimum energy, by
considering several constructions and taking the best; and (ii) proving an ansatz-free lower
bound, in which the correction due to positive $h$ agrees -- up to a parameter-independent
prefactor -- with the upper bound. (This summary is slightly misleading: in fact the thickness
enters our bounds through a nondimensional ratio that also involves the gravitational force,
the curtain's height, and the length scale of the wrinkling imposed at the top; moreover our
results are not asymptotic as $h \rightarrow 0$ -- rather, they require mainly that $h$ be
small compared to the length scale of the wrinkling imposed at the top. For a precise
statement of our bounds see Section \ref{sec:result}.)

The coarsening of folds or wrinkles in a sheet under tension has also been considered in the
physics literature
\cite{bib-mahacerda-draping,bib-benny-fissioning,bib-benny-raft,bib-benny-buildingblocks,bib-romandrapes}.
Our work is especially strongly connected with the treatment of ``heavy sheets'' in
\cite{bib-romandrapes}, which uses a self-similar construction whose basic cell is called a
``wrinklon.'' One of the constructions used for our upper bound involves coarsening of the
folds with no lateral spreading; it is essentially the same as the deformation studied in
\cite{bib-romandrapes} (see Remark \ref{link-to-vandeparre-etal}). Our results
complement and extend those of that paper, by showing that the self-similar
coarsening considered there achieves the optimal scaling law in a certain parameter
regime, and by showing that other constructions (involving
lateral spreading of the sheet) do better in other regions of parameter space. We note in
passing that \cite{bib-romandrapes} also discusses the coarsening of folds in ``light
sheets.'' An analysis of that problem in the spirit of the present paper has (to the best of
our knowledge) not yet been done, though the blistering problem considered in
\cite{bib-blisters-linear,bib-blisters,bib-jin-sternberg} seems closely related.

In focusing on the minimum energy, we obtain something like a phase diagram. Each of the
constructions used for the upper bound has a different energy scaling law, which can be
expressed in terms of two nondimensional parameters (see \eqref{avgeps}). By considering
which construction is best, we divide the parameter space into regions according to the
character of the optimal construction. Since we obtain scaling laws not prefactors, our
analysis leaves some uncertainty about the exact locations of the ``phase boundaries.'' Our
lower bound assures that our list of constructions is complete, i.e. we have not forgotten any
phases. However our results address {\it only} the energy scaling law -- they do not exclude
the possibility that some other construction, qualitatively different from the ones considered
here, could also achieve the optimal scaling in some region of parameter space.

Mathematically speaking, there is nothing novel about using the best of several constructions
to obtain an upper bound on the minimum energy. The challenge of finding a lower bound that
matches (with respect to scaling) the upper bound is however far from routine. Something
similar has been achieved in model problems motivated by twinning in martensite
\cite{bib-bob+muller1}, uniaxial ferromagnets \cite{bib-choksikohnotto,bib-knupfermuratov},
the intermediate state of a type-I superconductor \cite{bib-choksi-conti-bob-otto}, structural
optimization \cite{kohn2014optimal,kohn2015optimal}, tension-driven wrinkling in thin elastic
sheets \cite{bib-bellakohn2}, and compression-driven wrinkling in thin elastic sheets
\cite{bib-bellakohn1,bib-blisters-linear,bib-blisters,bib-jin-sternberg,bib-kohn-nguyen}. All
these papers consider nonconvex variational problems regularized by higher order singular
perturbations, which develop microstructure as the coefficient of the regularizing term (call
it $h$) tends to zero. There is a general framework, known as ``relaxation,'' for finding the
limiting energy as $h \rightarrow 0$ \cite{bib-dacorogna-directmethods}. While there is as yet
no general framework for analyzing the correction to the energy associated with positive $h$,
some principles are beginning to emerge \cite{bib-kohn-icmproceedings}; one is the importance
of interpolation inequalities, which play a major role in many of the articles just cited and
also in the present work. In some settings the relaxed problem is very degenerate, providing
little guidance about the character of the microstructure. In the setting of this paper (as in
\cite{bib-bellakohn2}) the relaxed problem has a unique solution, which determines the
direction (but not the length scale) of the wrinkles. The nondegeneracy of the relaxed problem
makes the tension-driven wrinkling problems considered here and in \cite{bib-bellakohn2} quite
different from the compression-driven wrinkling problems considered in
\cite{bib-bellakohn1,bib-blisters-linear,bib-blisters,bib-jin-sternberg,bib-kohn-nguyen}.

Our work is closely related to contemporary work by the second author and Hoai-Minh Nguyen
\cite{bib-kohnnguyen-raft} concerning the wrinkling seen in a confined floating sheet
\cite{bib-benny-raft}. That problem has many similarities to the one considered here: the
sheet is in tension, the wrinkling is induced by lateral confinement, and the length scale of
the wrinkling is smallest at the unconfined edges of the sheet. There is an underlying energy,
considered already in \cite{bib-benny-raft}. The analysis in \cite{bib-kohnnguyen-raft}
adopts a variational viewpoint similar to that of the present paper, proving matching (with
respect to scaling) upper and lower bounds on the minimum energy. The construction leading to
the upper bound uses self-similar coarsening, and is similar to the ``type I deformation''
discussed in Section \ref{sectTypeI}. The proof of the lower bound shows that changing the length
scale costs energy, using an argument closely related to this paper's Lemma \ref{mainlemma}.
There is, however, a
key difference between the hanging drape considered here and the confined floating sheet
considered in \cite{bib-kohnnguyen-raft}, namely: the hanging drape can relieve its
confinement by lateral spreading, whereas in the floating sheet this is prohibited by the
boundary condition. This extra freedom means that we must consider spreading as well as
coarsening in connection with the upper bound, and it means that many new arguments are needed
for the lower bound. (The analysis of the floating sheet also has complications not present
here; in particular, the scale of the wrinkling at the unconfined boundary is determined by
energy minimization rather than being imposed as a boundary condition.)

It is tempting to think that the energy-minimizing deformations should resemble the ones used
to prove the upper bound. This is the principle, for example, behind the argument in
\cite{bib-romandrapes} concerning how the length scale of the folding in a hanging drape
varies with height. As noted above, however, our rigorous results concern only the energy, not
the spatial structure of the energy-minimizing pattern. There are in fact relatively few
problems involving microstructure where the optimal patterns are understood. One is
\cite{bib-contibranching}, which considers a model for the refinement of martensite twins near
an interface with austenite; another is \cite{bib-bella-transition}, which considers a model
for the behavior of a sheet near the boundary of a wrinkled region.

The preceding paragraphs fall far short of a comprehensive review concerning the mechanics of
thin sheets and the associated mathematical challenges. With respect to wrinkling or folding,
we have emphasized work in which the length scale of the microstructure varies with position,
thereby omitting many recent contributions including
\cite{bib-bedrossiankohn1,bib-bennynew,bib-benny1,bib-king-pnas,bib-schroll-2013prl}. Moreover
we have completely omitted other aspects of thin elastic sheets, such as the formation of
localized defects. For a broader review concerning the mechanics of thin sheets we refer to
\cite{bib-audolypomeau}. We also mention in passing some more mathematical work concerning the
energy scaling laws of specific defects in thin sheets
\cite{bib-brandman-kohn-nguyen,bib-conti,bib-olbermannmuller-dcone,bib-venkataramani}.

Returning to our hanging drape, we now offer some heuristics to motivate the analysis that
begins in Section \ref{sec:model}. Here and in the rest of the paper, we often speak of
``wrinkles'' rather than ``folds.'' This is because we do not expect (and our model does not
predict) sharp folds similar to creases in a piece of paper; rather, we expect the
out-of-plane profile of the drape to be smooth (perhaps approximately sinusoidal). Also, we
often speak of a ``sheet'' rather than a ``drape,'' since we model the drape as a thin elastic
sheet.

Recall that the boundary condition at the top involves small-scale wrinkling, which costs
bending energy. If the thickness of the sheet is large enough to make the bending resistance
important, then (as noted earlier) one expects to see less bending far from the top. This can
be achieved by two rather different mechanisms: coarsening of the folds or spreading of the
sides. A third option -- compression of the sheet -- is not anticipated since buckling is
energetically preferred over compression. Our results support the intuition that there should
be no compression, since our upper bound uses compression-free constructions while our
matching lower bound has no such hypothesis. (They do not, however, show that the minimizer is
compression-free.)

Gravitational effects oppose both the coarsening of the folds and the lateral spreading of the
sheet. Indeed, gravity pulls vertically, favoring a configuration that hangs straight down.
Coarsening increases the amplitude of the out-of-plane displacement, while spreading involves
horizontal deformation; the presence of either mechanism works against the effects of
gravity.

The deformation of the hanging drape reflects the competition between these effects. A gradual
deviation from ``hanging straight down'' is preferred (by the gravitational effects) over an
abrupt deviation. Therefore we expect the wrinkling to gradually coarsen and the sheet to
gradually get wider as the distance from the top increases. Our upper bounds are consistent
with this expectation; moreover they provide guidance (consistent with that in
\cite{bib-romandrapes}) concerning the rate at which coarsening occurs, and they identify the
parameter regime in which spreading affects the scaling law. However while the spreading of a
real drape is smooth, the spreading in our constructions is only piecewise smooth. This is
convenient, since it makes it easier to estimate the elastic energy; and it is permissible,
since our results concern the energy scaling law (not the prefactor, and not the character of
the optimal deformation).

The article is organized as follows. Sections~\ref{sec:model}-\ref{sec:bulk} are all in some
sense introductory: Section~\ref{sec:model} presents our model and discusses
in general terms the properties of the low-energy configurations; Section~\ref{sec:result}
presents our main result -- the matching (with respect to scaling) upper and lower
bounds -- and briefly discusses when (in terms of the physical parameters) lateral spreading
affects the energy scaling law; Section~\ref{sec:rescaling} rescales the energy to
decrease the number of independent parameters; and Section~\ref{sec:bulk} discusses
the ``bulk energy,'' which plays the role of a relaxed problem (its minimum is the
limiting value of the energy when $h \rightarrow 0$). With those preliminaries in place, we turn in
Section~\ref{sec:ub} to the upper bound, which is proved by considering several candidate
deformations. Finally, Section~\ref{sec:lb} presents our lower bound, which is mathematically
speaking the subtlest aspect of the paper. A sketch of the main ideas underlying the lower
bound is given in Section~\ref{subsec:7-1}.
\bigskip

{\bf Notation.} We shall denote by $C$ a generic constant, i.e., a constant whose value may
change throughout the computation. The symbols $\sim$, $\lesssim$, and $\gtrsim$ indicate that
the estimates hold up to a finite universal multiplicative constant $C$, e.g., $a \lesssim b$
stands for $a \le Cb$. The tensor product $u \otimes v$ is defined as the $3 \times 3$ matrix
that is component-wise defined by $(u \otimes v)_{ij} = u_iv_j$. When $f(x,y)$ is a function,
we often use subscripts to denote partial derivatives; for example
$f_{,x}=\partial f/\partial x$ and $f_{,xy}= \partial^2f/\partial x \partial y$.
Finally, when $a$ and $b$ are real
numbers, we write $a \wedge b$ for the minimum of $a$ and $b$.


\section{The model}\label{sec:model}

In this section
we discuss the domain, the energy functional, and the boundary condition at the top.
We also discuss in general terms the expected behavior.

We assume the drape has (a small) thickness $h > 0$ and a rectangular shape of width $2W$ and
length $L$ (we choose width $2W$ in order to have a symmetric domain $[-W,W]$). We denote the
domain by $\Omega := [-W,W] \times [-L,0]$. The curtain is clamped at the top $\Gamma_T := \{
(x,y) \in \Omega : y=0 \}$ while it is free to move elsewhere. The wrinkles
prescribed at $\Gamma_T$ will have a (small) wavelength $w_0$; their shape will be
specified in a moment (see \eqref{bdry}). For simplicity we assume that $W = k w_0$ for
some (typically large) integer $k$.

As usual in elasticity, the stable configurations of the drape are local minima of
an ``energy functional,'' obtained by adding the elastic energy and the work done by gravity.
Our goal (as discussed in the Introduction) is to understand how the minimum energy scales with
respect to the physical parameters. Evidently, we are studying the energy of the ground state.

For the elastic energy we use a {\it geometrically linear F\"oppl-von K\'arm\'an} model, and we take
Poisson's ratio to be zero. This is, admittedly, a qualitatively accurate model not a quantitatively
accurate one: for real materials Poisson's ratio is usually not zero, and the F\"oppl-von K\'arm\'an
framework is only appropriate when the out-of-plane deformations have small slope. We believe, however,
that our choice captures the essential physics of the phenomena we wish to study. This view
is supported by the mechanics and physics literature on wrinkling, where the
F\"oppl-von K\'arm\'an framework is widely used. It is also supported by the
mathematics literature on thin elastic sheets, where energy scaling laws initially derived
using a F\"oppl-von K\'arm\'an model have been shown to hold also in more nonlinear settings, see
e.g. \cite{bib-blisters,bib-conti}. (For further discussion about the appropriateness of the
F\"oppl-von K\'arm\'an framework see e.g. \cite{bib-venkataramani}.)
Based on the preceding considerations, our energy functional is:
\begin{equation}\label{energy}
 E_h ( u,\xi ) = \iint_\Omega |e(u) + \frac{1}{2} \nabla \xi \otimes \nabla \xi|^2 + h^2 |
 \nabla^2 \xi |^2 \dxy + \tau \iint_\Omega u_y \dxy,
\end{equation}
where $u = (u_x,u_y)$ denotes the in-plane displacement, $\xi$ is the out-of-plane
displacement, $e(u) = \frac{\nabla u + \nabla u^T}{2}$ is the symmetric gradient of $u$, and
$\tau > 0$ is a given parameter (the ratio between the gravitational constant and Young's
modulus of the elastic material). For our geometrically linear F\"oppl-von K\'arm\'an model
to be reasonable we want the curvature to be much smaller than $1/h$. This means that any
length scale in the deformation (in particular the period $w_0$ of the prescribed wrinkling
at the top) should be larger than $h$.

We turn now to the boundary condition $u(x,0)$ and $\xi(x,0)$ imposed
at $\Gamma_T$ (the top of the sheet). The out-of-plane deformation $\xi$ should be
periodic with period $w_0$ (note that since we assume $W=kw_0$, $\Gamma_T$ is filled by
exactly $k$ periods). Moreover we want to avoid strain in the horizontal
direction, and the deformation should achieve a specified overall horizontal
compression factor $\Delta$. Finally, it is natural to choose $u(x,0)$ and $\xi(x,0)$
so that the bending energy is minimized subject to these constraints.
Focusing initially on the two-period interval $[-w_0,w_0]$, we seek
$u_0,\xi_0 : [-w_0,w_0] \to \mathbb{R}$ such that
\begin{gather}\label{incompressible}
u_{0,x}(x) + \xi_{0,x}^2(x)/2 = 0, \nonumber \\
u_0(-w_0) - u_0(w_0) = 2 \Delta w_0, \nonumber
\end{gather}
and such that the bending energy
\begin{equation}\nonumber
 h^2 \int_{-w_0}^{w_0} \xi_{0,xx}^2(x) \dx
\end{equation}
is minimized subject to these constraints. Using the method of
Lagrange multipliers one finds that the choice
\begin{equation}\nonumber
 \xi_0(x) = \frac{w_0\sqrt{\Delta}}{\pi}\sin (2\pi x/w_0)
\end{equation}
is optimal. Extending $\xi_0$ by periodicity, we are led to impose the boundary condition that
\begin{equation}\label{bdry}
 \xi(x,0) := \frac{w_0 \sqrt{\Delta}}{\pi} \sin (2\pi x w_0^{-1}),\quad u_x(x,0) :=
 -\frac{1}{2} \int_0^x |\xi_{,x}(t,0)|^2 {\ud t},\quad u_y(x,0)=0
\end{equation}
at $\Gamma_T$.

The elastic energy of our boundary condition is
$h^2 \int_{-W}^W \xi_{,xx}^2 \dx = 2Wh^2 \Delta (8\pi^2 w_0^{-2})$. For the trivial planar
deformation $\xi(x,0) \equiv 0, u_x(x,0) = -\Delta x$ the elastic energy
$\int_{-W}^W u_{,x}^2 \dx$ is $2W \Delta^2$. Our boundary condition has lower energy than the
trivial one when
\begin{equation} \label{hwdelta}
h \le w_0 \sqrt{\Delta} (8\pi^2)^{-1/2}.
\end{equation}
Thus it is reasonable to prescribe~\eqref{bdry} only if~\eqref{hwdelta} is satisfied.

\subsection{The expected behavior}

We now discuss the expected form of a deformation with small energy. This is, in effect, a
description of the deformation associated with our upper bound -- which achieves the optimal
scaling, according to our lower bound.

Since we do not prescribe boundary conditions on the lateral part of the boundary, the sheet is free
to get wider (and, therefore, to relax part of the confinement forced by the boundary
conditions at the top). Using two different constructions, we will show that the energy
required to significantly change the value of $u_x$ (i.e. to release the lateral confinement)
over the length $\l$ from the top of the sheet is at most of order $W^2\Delta^2
\min(W\l^{-1},W^3\l^{-3})$. The energy required to significantly decrease the amplitude of the
out-of-plane displacement $\xi$ scales like $W(\Delta w_0^2) \tau L \l^{-1}$. If the sheet releases the
confinement, then both of these terms contribute to the energy.

Assuming the sheet avoids compression, besides getting wider towards the bottom
it needs to waste some arc length in the horizontal direction in the region
where it is confined. This can be done by buckling out of the plane in wrinkles
with an ``average period'' $w=w(y)$. The boundary condition at
$\Gamma_T$ sets $w(0) = w_0$. The bending energy of a deformation consisting of wrinkles with
period $w$ is $16\pi^2W\Delta h^2 w^{-2}$. It is obvious that this term prefers to increase
the period $w$ as fast as possible towards the bottom of the sheet (away from $\Gamma_T$), but
a rapid change in the length scale $w$ would require a large change in the amplitude of the
wrinkles (large $\xi_{,y}$). Since the sheet is stretched in the vertical direction by gravity,
the part of the energy coming from stretching in the vertical
direction prefers small $\xi_{,y}$. In the end, the competition between these two preferences
determines the rate at which the length scale $w(y)$ increases.

In our constructions, the variation of $w(y)$ is achieved using ``building blocks.''
Each building block is a deformation defined on a rectangle of width
$w$ with sinusoidal profiles at the top and bottom boundaries with period $w$ and $3w$
respectively. The idea of using such building blocks was already present in
\cite{bib-blisters-linear,bib-benny-fissioning,bib-jin-sternberg}; our building blocks are
called wrinklons in \cite{bib-romandrapes}. As we'll see in Section~\ref{sec:ub}, energy
minimization requires the height $l$ and width $w$ of a building block to be related by
$l \sim w^2 \sqrt{\tau L} h^{-1}$.

The optimal number of building blocks -- i.e. the number of generations of coarsening -- depends
on several parameters. It is an increasing function of $L$ (in a longer sheet there is room
for more generations of coarsening), and an increasing function of $h$ (a thicker sheet is
harder to bend, so it decreases the amount of bending faster).
On the other hand, the number of building blocks is a decreasing function of $w_0$ (finer
wrinkles require more bending, so a sheet with finer wrinkles at the top prefers to coarsen
the wrinkles faster), and a decreasing function of $\tau$ (in a heavier drape the effect of gravity
is stronger, so out-of-plane displacement is more expensive, which implies that
coarsening is also more expensive).

If the sheet is very long, the coarsening process may finish at a height above the bottom of the
drape. As one goes further down, the out-of-plane displacement becomes an affine function with
the correct slope. From this point on, there are no more contributions to the energy
and the sheet does not change its shape (except for the vertical
deformation due to stretching).

Our discussion has emphasized the properties of the deformation in the horizontal direction. In
the vertical direction the situation is much simpler. Since the sheet is pulled down by
gravity independently of the horizontal position, we expect the vertical deformation to be
independent of $x$ (up to a small correction due to the wrinkling). In
Section~\ref{sec:bulk} we will formulate a one-dimensional variational problem which will be
used to identify the main part of the optimal vertical displacement $u_y$.

\section{The main result}\label{sec:result}

This section
presents our main result, and provides some discussion to help elucidate its
consequences.

\begin{theorem}\label{thm1}
Assume
\begin{gather}\label{tauL}
\tau L \ge 4,\\
h\Delta^{-1/2} \le w_0 \le (2\cw) W,\label{wgeh2}
\end{gather}
where $\cw > 0$ is a small universal constant. Then there exist universal
constants $C_{UB} > C_{LB} > 0$ such that for any deformation $(u,\xi)$ which
satisfies~\eqref{bdry} we have
\begin{equation}\label{result}
-\frac{1}{12}\tau^2L^3(2W) + C_{LB} \epsilon \le
\min_{(u,\xi)} E_h(u,\xi) \le -\frac{1}{12}\tau^2L^3(2W) + C_{UB} \epsilon
\end{equation}
where
\begin{multline}\label{epsilon}
\epsilon := W\Delta \min\Bigg[
h \sqrt{\tau L} \log \left( w_0^{-2} \left( \frac{h L}{\sqrt{\tau L}} \wedge 4W^2\right) + 1\right),\\
\min_{l \in (0,L)} \left\{ h \sqrt{\tau L} \log \left( \frac{h l}{w_0^2 \sqrt{\tau L}} + 1 \right) +
w_0^2 \tau L l^{-1} +
W\Delta \min \left( \left( \frac{W}{l}\right), \left( \frac{W}{l}\right)^3 \right) \right\} \Bigg].
\end{multline}
\end{theorem}

\begin{remark*}
The first inequality in~\eqref{wgeh2} is motivated by \eqref{hwdelta} (dropping the
constant $(8\pi^2)^{-1/2}$ compared to~\eqref{hwdelta} does not change the scaling
law). The second inequality in~\eqref{wgeh2} (with $c_w \ll 1$) says that the period
of the wrinkling prescribed at the top is much smaller than the width of the sheet.

We see that the minimum of the energy in Theorem~\ref{thm1} consists of two parts. One is
the ``bulk energy'' $-\frac{1}{12}\tau^2L^3(2W)$, which comes (as we'll show in
Section \ref{sec:bulk}) from the stretching of the sheet in the vertical direction
on account of gravity. The other is the ``excess energy'' due to positive $h$; it is of
order $\epsilon$. We observe that $\epsilon$ does indeed vanish as $h\rightarrow 0$; this
is consistent with the fact that wrinkling uses less energy in a sheet of smaller thickness.
\end{remark*}

\begin{remark*}
The paper \cite{bib-romandrapes} distinguishes between ``heavy sheets'' and ``light sheets,''
and shows that the rate at which wrinkles coarsen is different in the two cases.
The assumption~\eqref{tauL} means that our drapes are ``heavy sheets.''
\end{remark*}

The excess energy $\epsilon$ is the minimum of two different terms -- the first is the
energy of the construction when the sheet does not (significantly) release the lateral
confinement; the latter is the energy of the construction when the sheet releases (much
of) its lateral confinement. It is natural to try to understand, in terms of the
physical parameters, the regimes in which one or the other
term dominates. To this end we introduce the non-dimensional parameters:
\begin{equation}\nonumber
\alpha := \frac{h L}{w_0^2 \sqrt{\tau L}}, \quad
\beta := \frac{w_0}{L} \sqrt{\tau L}, \quad
r := \frac{l}{L}.
\end{equation}
We focus on the case when the sheet is not extremely long, in the sense that
$\alpha = \frac{h L}{w_0^2 \sqrt{\tau L}} \le \left( \frac{2W}{w_0} \right)^2$. Then we get
the following formula for the ``average'' excess energy:
\begin{multline}\label{avgeps}
\frac{\epsilon}{LW \Delta}\\
= \min \left( \alpha \beta^2 \log(\alpha+1),
\min_{r \in (0,1)} \left\{ \alpha \beta^2 \left( \log(\alpha r + 1) + (\alpha r)^{-1} \right) +
\frac{W}{L} \Delta \min\left( \left(\frac{W}{Lr}\right), \left(\frac{W}{Lr}\right)^3 \right)
\right\} \right).
\end{multline}
We first neglect the last term
$\frac{W}{L} \Delta \min\left( \left(\frac{W}{Lr}\right), \left(\frac{W}{Lr}\right)^3 \right)$
and investigate when
\begin{equation}\label{invest1}
\min_{r \in (0,1)} \log(\alpha r + 1) + (\alpha r)^{-1} \le  \log(\alpha +1).
\end{equation}
Since $r \in (0,1)$, for \eqref{invest1} to hold, the function
$\mathcal{F}(t) = \log(t + 1) + t^{-1}$ must achieve its minimum at a value $t_{\rm min}$ that is
smaller than $\alpha$. Therefore~\eqref{invest1} can be satisfied only if
\begin{equation}\nonumber
\alpha > \frac{1+\sqrt{5}}{2}.
\end{equation}
Since the last term in~\eqref{avgeps} (the one we neglected) is positive, it is clear that
when $\alpha < (1+\sqrt{5})/2$ the first term in \eqref{epsilon} is larger, i.e. the sheet
does not prefer to release most of the lateral confinement. On the other hand, if
$\alpha$ is considerably larger than $(1+\sqrt{5})/2$ we expect the second part of
the right hand side in~\eqref{avgeps} to be smaller provided the neglected term
is sufficiently small, in which case most of the lateral confinement should be released.


\section{The rescaling}\label{sec:rescaling}

The scaling law~\eqref{result} depends on numerous parameters: the sheet's length $L$, its
width $W$, and its thickness $h$, the ``gravitational'' coefficient $\tau$, and the period of
the wrinkling at the top $w_0$. To simplify the analysis, it is convenient to reduce the number
of parameters by nondimensionalizing the problem and by using its special structure to eliminate
$\Delta$.

The effect of nondimensionalization is that it permits us to consider only the case $W=1/2$.
This is achieved by measuring length in units of width. Explaining in detail: if $(u,\xi)$ is
a deformation defined in $(-W,W) \times (-L,0)$, we define a new deformation $(v,\mu)$ by:
\begin{equation}\nonumber
 v(x,y) := (2W)^{-1} u(2Wx,2Wy), \qquad \mu(x,y) := (2W)^{-1} \xi(2Wx,2Wy).
\end{equation}
The deformation $(v,\mu)$ is defined in $(-1/2,1/2) \times (-L/(2W),0)$, and we have
\begin{equation}\label{nondim1}
E_{\tilde h,1/2,\tilde L,\tilde \tau, \tilde w_0} (v,\mu) = (2W)^{-2} E_{h,W,L,\tau,w_0}(u,\xi),
\end{equation}
where we have listed all relevant parameters for the energy as indices. The rescaled parameters are:
\begin{equation}\nonumber
\tilde h = (2W)^{-1}h, \tilde L = (2W)^{-1}L, \tilde \tau = (2W)\tau, \tilde w_0 = (2W)^{-1} w_0.
\end{equation}
This nondimensionalization replaces $L$, $h$, and $w_0$ (which are lengths) by
their quotients with $2W$; similarly, it replaces $\tau$ (which has dimension ${\rm length}^{-1}$) by its
product with $2W$. We observe that~\eqref{nondim1} is consistent with~\eqref{result}, \eqref{epsilon}.
Therefore to prove Theorem~\ref{thm1} it is sufficient to consider the case $W = 1/2$.

Our second reduction uses the special structure of our energy functional -- specifically,
the fact that the membrane term is quadratic in $u$ and quartic in $\xi$, while the bending term is
quadratic in $\xi$, and the gravitational term is linear in $u$ -- to eliminate the parameter $\Delta$.
Defining
\begin{equation}\nonumber
v(x,y) := \Delta^{-1} u(x,y), \qquad \mu(x,y) := \Delta^{-1/2} \xi(x,y),
\end{equation}
we see that
\begin{equation}\nonumber
E_{\tilde h, \tilde \tau, 1} (v,\mu) = \Delta ^{-2} E_{h,\tau,\Delta} (u,\xi),
\end{equation}
where $\tilde h = \Delta^{-1/2}h, \tilde \tau = \Delta^{-1} \tau$. This relation is again
consistent with Theorem~\ref{thm1}. Therefore in proving the theorem it is sufficient to consider $\Delta = 1$.

In the rest of the paper we will assume that $\Delta = 1$ and $W=1/2$. Note that in this case the second
hypothesis \eqref{wgeh2} of Theorem~\ref{thm1} becomes
\begin{equation}\label{wgeh}
h \le w_0 \le \cw.
\end{equation}

\section{The bulk energy}\label{sec:bulk}

The first term in our energy scaling law~\eqref{result} is the limiting value of the minimum
energy as $h \rightarrow 0$. While the proof of this
assertion is given later on (in the course of establishing our upper and lower bounds), the present section
lays necessary groundwork by considering a ``bulk energy'' that includes only vertical stretching and gravity.

To motivate the definition of the bulk energy, we begin by substituting $h=0$ in the definition of our
functional $E_h$. Gravity pulls downward, so we expect the vertical displacement to satisfy
$u_{y,y} \ge 0$, and assuming this we have
\begin{multline}\nonumber
E_0(u,\xi) \ge \iint_{\Omega} |u_{y,y}(x,y) + \xi_{,y}^2(x,y)/2|^2 \dxy + \tau \iint_\Omega u_y(x,y) \dxy \\
\ge \iint_{\Omega} \left|u_{y,y}\right|^2(x,y) + \tau u_y(x,y) \dxy.
\end{multline}
Therefore it is natural to consider the bulk energy functional, defined by
\begin{equation}\label{bulk}
B(f) := \iint_\Omega |f_{,y}|^2 + \tau f \dxy,
\end{equation}
where $f : [-L,0] \to \mathbb{R}$, $f(0)=0$ ($f$ plays the role of $u_y$).
It is easy to find the unique minimizer of $B$:
\begin{equation}\label{f}
f(y) = (\tau y^2 + 2\tau L y)/4, \qquad \min B = B(f) = -\frac{1}{12}\tau^2 L^3.
\end{equation}
The following estimate for $f_{,y}(y) = \tau(y+L)/2$
will be useful later:
\begin{equation}\label{estf}
\begin{array}{cl}
 f_{,y}(y) \ge \tau L/4   & \mbox{if }  y \in [-L/2,0],\\
 0 \le f_{,y}(y)  \le  \tau L/2 & \mbox{if }  y \in [-L,0].
\end{array}
\end{equation}
We see that to have $f_{,y}(y) \ge 1$ for $y \ge -L/2$ we assumed~\eqref{tauL}.
For $\bar u_y(x,y) := f(y), \bar u_x(x,y) := u_x(x,0), \bar \xi(x,y) := \xi(x,0)$ we have
\begin{multline}\nonumber
E_0(\bar u,\bar \xi) = \iint_\Omega |\bar u_{x,x} + \bar \xi_{,x}^2/2|^2 +
|\bar u_{x,y} + \bar u_{y,x} + \bar \xi_{,x}\bar\xi_{,y}|^2/2 + |\bar u_{y,y} + \bar\xi_{,y}^2 / 2|^2 +
\tau \bar u_y \dxy \\
= \iint_\Omega |\bar u_{y,y}|^2 + \tau \bar u_y \dxy = B(f);
\end{multline}
thus $\min_{(u,\xi)} E_0(u,\xi)$ (subject to $u_{y,y} \ge 0$) is attained and is equal to $B(f)$.
We will see later that for a deformation $(u,\xi)$ to nearly minimize $E_h$, the vertical part of the deformation
$u_y$ must be close to $f$.


\section{The upper bound}\label{sec:ub}

This section proves our upper bound for the minimum energy. As shown in Section \ref{sec:rescaling},
it suffices to consider $W=1$ and $\Delta = 1$. Our goal is therefore to show that
\begin{equation}\label{ub1}
 \min E_h(u,\xi) \le -\frac{1}{12}\tau^2 L^3 + C_{UB} \epsilon,
\end{equation}
provided~\eqref{tauL} and~\eqref{wgeh} hold, $(u,\xi)$ satisfy the prescribed boundary condition~\eqref{bdry},
and $\epsilon$ is defined by~\eqref{epsilon}.

The proof of~\eqref{ub1} uses three types of deformations.
Our type I deformation involves self-similar coarsening of the wrinkles, while the sheet remains horizontally
confined. The coarsening is achieved through multiple generations of building blocks; the period of the
wrinkles is tripled in each generation. If the vertical length of each generation is chosen in the optimal
way, its contribution to the energy is of order $h \sqrt{\tau L}$. If the sheet is very long -- more specifically,
if we reach a generation of wrinkles with period comparable to the width of the sheet -- then we change
(within one generation) the out-of-plane displacement to an affine function (since when $\xi$ is an affine
function with the correct slope all the terms contributing to the excess energy vanish). This construction
involves no lateral spreading: the values of $u_x(-1/2,y)$ and $u_x(1/2,y)$ are independent of $y$. The
vertical deformation $u_y$ agrees with the minimizer $f$ of the bulk energy $B$.

To be energetically efficient, the coarsening process associated with a type I deformation needs some room. We
discuss the case when $L$ is too small for coarsening at the end of Section \ref{sectTypeI}. Since there is no
room for coarsening, in this setting the type I deformation keeps the profile of the wrinkling independent of $y$.

Our second type of deformation is a modification of the first. It also involves self-similar coarsening
of the wrinkles, but in contrast with the first type the horizontal confinement is relaxed at some point.
To describe it, consider a type I deformation, and choose a particular generation in the coarsening process.
We denote by $n$ the order of this generation (i.e. there are $n-1$ generations above) and by $l_n$ the
vertical length of the building blocks in this generation. Our type II deformation is
identical with the type I deformation through the first $n-1$ generations of the coarsening process,
but different starting at the $n$-th generation; it completely releases the horizontal
confinement of the sheet within the $n$-th generation (the value of $u_x$ at the extremes $\pm 1/2$
changes from $\pm 1/2$ at the top of the $n$th generation to $0$ at the bottom).

Full details of the type II deformation are given in Section~\ref{sectTypeII}, but here is a sketch.
If we keep $u_y = f$, changing $u_x$ from order $1$ to $0$ over the length $l_n$ results in the
term $|u_{x,y} + u_{y,x} + \xi_{,x}\xi_{,y}|^2$ being of order $l_n^{-2}$ (since $u_{y,x} = 0$
and the term involving derivatives of $\xi$ is not larger than $u_{x,y}$). Integrating this
term over the domain of size $l_n$, we find that this mechanism for releasing the horizontal confinement
has an energetic cost of order $O(l_n * (1/l_n)^2) = O(l_n^{-1})$.
A different possibility is to set $u_{y,x} := -u_{x,y} - \xi_{,x}\xi_{,y}$. This obviously
eliminates the term $|u_{x,y} + u_{y,x} + \xi_{,x}\xi_{,y}|$ from the energy, but it increases the
term $|u_{y,y} + \xi_{,y}^2/2|$. For this alternative mechanism of releasing the horizontal
confinement the energetic cost turns out to be $O(l_n^{-3})$.

Our type III deformation is similar to
a type II deformation with $n=0$. It flattens the sheet and eliminates the horizontal confinement
over an interval of height starting near $y=0$. As with a type II deformation, there are two
versions of this construction, corresponding to different choices about which membrane term should be made
to vanish.

\subsection{The construction of a building block and the type I deformation}\label{sectTypeI}

In this section we construct a deformation consisting of several generations of building blocks, with
the length scale of wrinkling being tripled in every generation. We first define a building block,
which is a deformation $(v,\mu)$ (the horizontal and the out-of-plane displacement) defined on
$[0,1] \times [0,1]$ with wrinkles with period $1/3$ at the top and with period $1$ at the bottom
with periodic lateral boundary conditions (see~\cite{bib-benny-buildingblocks} for a numerical
study of the optimal shape for one such building block). Our Type I deformation is then obtained
by patching together rescaled versions of the building block (the boundary conditions for the
building block were chosen to make this possible).

The building block is a deformation $(v,\mu) : [0,1]\times [0,1] \to \R^2$. It has a sinusoidal profile
of period $1/3$ in a neighborhood of $y=0$, and a sinusoidal profile of period $1$ in a neighborhood of $y=1$.
There is neither compression nor tension in the $x$-direction, and both the membrane and bending energy are
finite. Saying the same in mathematical terms: both $x+v(x,y)$ and $\mu(x,y)$ should be $1$-periodic in $x$,
and they must satisfy:
\begin{gather}
\mu(x,y) = \frac{1}{3\pi} \sin(6\pi x), \quad
v(x,y) = -\frac{1}{2}\int_0^x |\mu_{,x}(t,0)|^2 \ud t, \qquad x\in [0,1], y \in [0,1/4],\label{block1}\\
\mu(x,y) = \frac{1}{\pi} \sin(2\pi x), \quad
v(x,y) = -\frac{1}{2}\int_0^x |\mu_{,x}(t,1)|^2 \ud t, \qquad x\in [0,1], y \in [3/4,1],\label{block2}\\
v_{,x}(x,y) + \frac{1}{2} |\mu_{,x}(x,y)|^2 = 0 \textrm{ for } (x,y) \in [0,1]^2,\quad
|\mu_{,y}| \le 1,\quad |\mu|\le1 \label{block3}\\
E_m := \iint_{(0,1)^2} |v_{,y} + \mu_{,x} \mu_{,y} |^2 + |\mu_{,y}|^2 \dxy < \infty  , \quad
E_b := \iint_{(0,1)^2} |\nabla^2 \mu|^2 \dxy < \infty.\label{block4}
\end{gather}
Such a building block can be obtained as follows: let $g_1, g_2 : [0,1] \to [0,1]$ be smooth functions
which satisfy
\begin{gather}\label{blockdef1}
\begin{array}{ll}
g_1(y) = 1, & y \in [0,1/4],\\
g_2(y) = 1, & y \in [3/4,1],
\end{array}\\
\label{blockdef3}
g_1^2(y) + g_2^2(y) = 1, \quad y \in [0,1],\\
|g_1'(y)| + 3|g_2'(y)| \le 3\pi, \quad y \in [0,1].\label{blockdef4}
\end{gather}
We define
\begin{equation*}
\mu(x,y) := \frac{g_1(y)}{3\pi} \sin (6\pi x) + \frac{g_2(y)}{\pi} \sin (2\pi x),\quad
v(x,y) := -\frac{1}{2}\int_0^x |\mu_{,x}(t,y)|^2 \ud t, \quad (x,y) \in [0,1]^2.
\end{equation*}
Using the properties of $g_1$ and $g_2$ it is easy to verify (\ref{block1}--\ref{block4}).
Indeed, \eqref{block1} and \eqref{block2} follow from~\eqref{blockdef1} and \eqref{block3}
follows from~\eqref{blockdef3} and \eqref{blockdef4}. Finally, \eqref{block4} is a direct
consequence of the smoothness of $g_1$ and $g_2$.

To define the type I deformation we patch together rescaled versions of the building block,
and use $u_y(x,y) := f(y)$ for the horizontal displacement. The new rescaled displacement
$u_x, \xi$ (defined on a rectangle of size $\omega \times l$ with the upper left corner at
$(x_0,y_0)$) is defined as:
$$
u_x(x,y) := x_0 + \omega v\left(\frac{x-x_0}{\omega}, -\frac{y-y_0}{l}\right), \quad
\xi(x,y) := \omega \mu\left(\frac{x-x_0}{\omega}, -\frac{y-y_0}{l}\right).
$$
From $|\mu|\le 1$ we get $|\xi| \le w$. The period of the wrinkles at the bottom and upper end
of this building block are $\omega$ and $\omega/3$, respectively. We also have
\begin{gather}\nonumber
u_{x,x} + \xi_{,x}^2 / 2 = v_{,x} + \mu_{,x}^2 / 2 = 0,\\ \nonumber
u_{x,y} + u_{y,x} + \xi_{,x}\xi_{,y} = \frac{\omega}{l} \left( v_{,y} + \mu_{,x}\mu_{,y}\right), \\ \nonumber
\left(u_{y,y} + \xi_{,y}^2 / 2\right)^2 - |f_{,y}|^2 = f_{,y} \xi_{,y}^2 + \xi_{,y}^4 / 4 =
(\omega/l)^2 \left( f_{,y} \mu_{,y}^2 + (w/l)^2  \mu_{,y}^4 / 4 \right),\\
\xi_{,xx} = w^{-1}\mu_{,xx}, \quad \xi_{,xy} = l^{-1}\mu_{,xy}, \quad \xi_{,yy} = w l^{-2}\mu_{,yy}. \nonumber
\end{gather}
In what follows we will always keep $\omega \le l$. Then~\eqref{block3} implies
$|(\omega/l)^2 \mu_{,y}^4/4| \le \mu_{,y}^2 / 4$ and $|\nabla^2 \xi|^2 \le \omega^{-2} |\nabla^2 \mu|^2$,
and the elastic energy of one building block (over the rectangle of size $\omega \times l$) is bounded by
\begin{equation}\label{blockenergy}
2 \omega l \left[ \left(\frac{\omega}{l}\right)^2 (1 + f') E_m + h^2 \omega^{-2} E_b\right] +
\iint (f'^2 + \tau f).
\end{equation}
When summed over the whole domain $\Omega$ the second term becomes exactly~\eqref{bulk}. We are
interested in the first term which is the excess energy due to positive $h$ (roughly speaking: the
energy due to wrinkling).

We will patch together several generations of building blocks. Let $N$ be the number of generations
(to be chosen later), and let $l_n$ and $\omega_n := 3^n w_0$ for $n=1,\dots,N$ be the length and
width of the building block in the $n$-th generation, respectively. We may assume without loss of
generality that $3^K w_0 = 1/2$ for some integer $K$.\footnote{Indeed, if (before nondimensionalization)
there is a $K$ such that $3^{K-1} w_0 < W/2 < 3^K w_0$, then we can consider a slightly wider drape,
of width $\overline{W}$ such that $\overline{W}/2 = 3^K w_0$ (for an upper bound it does no harm
to increase the domain). After nondimensionalization (which divides all lengths by $\overline{W}$), we
get a problem in which $3^K w_0 = 1/2$.}
Since one block has width $\omega_n$ and the width of the sheet is $1$, each generation has
$\omega_n^{-1}$ identical building blocks. We sum the first term in~\eqref{blockenergy} over all
building blocks to get
\begin{equation}\label{sumenergy}
\sum_{n=1}^N \frac{1}{\omega_n} 2\omega_n l_n
\left[ \left(\frac{\omega_n}{l_n}\right)^2 (1 + f')E_m + h^2 \omega_n^{-2} E_b\right] \le
2\sum_{n=1}^N \left[ \frac{\omega_n^2}{l_n} \tau L E_m + h^2 \omega_n^{-2}l_n E_b\right],
\end{equation}
where we used~\eqref{estf} and~\eqref{tauL} to obtain $1 + f' \le 1 + \tau L /2 \le \tau L$.

We know that for \eqref{sumenergy} to be small we want the two terms on the right-hand side to be of
similar value, i.e. $\frac{\omega_n^2}{l_n} \tau L E_m \sim h^2 \omega_n^{-2}l_n E_b$.
From $\omega_n = 3^n w_0$ (the period is tripled in each generation) we obtain
$l_n \sim 9^n w_0^2 \sqrt{\tau L} \sqrt{\frac{E_m}{E_b}} h^{-1}$. Motivated by this we set
\begin{equation}\label{ln}
l_n := w_n^2 \sqrt{\tau L} h^{-1} = 9^n w_0^2 \sqrt{\tau L} h^{-1}.
\end{equation}
Using~\eqref{wgeh} we can now verify the previously used assumption $\omega_n \le l_n$:
\begin{equation}\label{lnwn}
l_n / \omega_n = 3^n \sqrt{\tau L} w_0 / h \ge \sqrt{\tau L} \ge 2.
\end{equation}

Let us first assume that the sheet is not very long in the sense that
\begin{equation}\label{notverylong}
 \sum_{n=1}^K l_n \ge L.
\end{equation}
Then the length of the sheet $L$ can be expressed as the sum of the lengths of all
generations of building blocks, i.e. $\sum_{n=1}^N l_n \approx L$. Let us define $N$
to be the smallest integer such that $\sum_{n=1}^N l_n \ge L$. It follows from \eqref{notverylong}
that $N \le K$ and so $w_n \le 1$ for $n=1,\ldots,N$. Using~\eqref{ln} we obtain
\begin{equation}\label{N}
 N = \left\lceil \log_9 \left( \frac{8}{9}\frac{h L}{w_0^2 \sqrt{\tau L} } + 1\right)\right\rceil \sim
 \log \left( \frac{hL}{w_0^2 \sqrt{\tau L}} + 1\right),
\end{equation}
where we assumed that
\begin{equation}\label{ub-hlarger}
 h \ge w_0^2 \sqrt{\tau L}/L.
\end{equation}
From the definition of $l_n$ we compute that one generation of wrinkles ($w_n^{-1}$ identical
building blocks) costs $C h\sqrt{\tau L}$. Therefore, for $L$ not too large the energy~\eqref{sumenergy}
is bounded by $Ch \sqrt{\tau L} \log \left( \frac{h L }{w_0^2 \sqrt{\tau L}} + 1 \right)$ and
\begin{equation}\label{boundA0}
\min E_h(u,\xi) \le
-\frac{1}{12} \tau^2 L^3 + Ch \sqrt{\tau L} \ln \left( \frac{hL}{w_0^2\sqrt{\tau L}} + 1 \right).
\end{equation}

Now we treat the case when \eqref{notverylong} is false, i.e., $L_0 := \sum_{n=1}^K l_n < L$.
For $y \in [-L_0,0]$ we define the deformation the same way as before. For that we need $K$
generations, where each contributes to the excess energy by a multiple of $h\sqrt{\tau L}$.
Since $K$ is defined through $3^K w_0 = 1/2$, the excess energy for this part will be of
order $h \sqrt{\tau L} \log(w_0^{-1})$. To define
the deformation for $y \in [-L,-L_0]$, we first observe that at $y=-L_0$ we have
\begin{equation}
\xi(x,-L_0) = \frac{1}{\pi}\sin(2\pi x), \qquad
u_x(x,-L_0) = -\frac{1}{2} \int_{0}^x \xi_{,x}^2(t,-L_0) \ud t, \quad x \in [0,1].
\end{equation}
We set $l_{K+1} := \sqrt{\tau L}/h$, and for $x\in[0,1], y \in [-L,L_0)$ we define
\begin{equation}
\xi(x,y) = g_1\left(-\frac{y+L_0}{l_{K+1}}\right)\frac{\sin(2\pi x)}{\pi} +
g_2\left(-\frac{y+L_0}{l_{K+1}}\right)\sqrt{2}x,\quad
u_x(x,-L_0) = -\frac{1}{2} \int_{0}^x \xi_{,x}^2(t,-L_0) \ud t,
\end{equation}
where $g_1,g_2$ are functions $g_1,g_2$ from (\ref{blockdef1}-\ref{blockdef4}), extended
respectively by $0$ and $1$ into $[1,\infty)$. With the above choice of $l_{K+1}$ it is easy
to compute that the contribution to the excess energy in $[-(L_0+l_{K+1}),-L_0]$ is of order
$h\sqrt{\tau L}$, while for $y < -(L_0+l_{K+1})$ we have $\xi(x,y) = \sqrt{2}x$, $u_x(x,y)=x$,
so there is absolutely no contribution to the excess energy.
We see that if $L$ is larger than $\sum_{n=1}^K l_n$, the contribution to the excess
energy is at most
\begin{equation}
Ch \sqrt{\tau L} \ln \left( \frac{1}{w_0^2} + 1 \right).
\end{equation}
Hence, together with \eqref{boundA0} we get that
\begin{equation}\label{boundA1}
\min E_h(u,\xi) \le -\frac{1}{12} \tau^2 L^3 +
Ch \sqrt{\tau L} \ln \left( w_0^{-2} \left( \frac{hL}{\sqrt{\tau L}} \wedge 1 \right) + 1 \right).
\end{equation}

If \eqref{ub-hlarger} does not hold, i.e. if $h < w_0^2 \sqrt{\tau L}/L$,
it should be better energetically to just propagate the deformation prescribed at $\Gamma_T$.
In this case we set
\begin{equation}\label{just-propagate}
 u_x(x,y) := u_{x,0}(x) = u_x(x,0), \quad u_y(x,y) := f(y), \quad \xi(x,y) := \xi_0(x) = \xi(x,0),
\end{equation}
and we get the total energy bounded by
\begin{equation}\label{boundA2}
 \min E_h(u,\xi) \le -\frac{1}{12}\tau^2 L^3 + C h^2w_0^{-2}L.
\end{equation}
Finally, since $\ln(1+t) \ge t/4$ for $t \in (0,1)$ and $\frac{hL}{\sqrt{\tau L}} < 1$ follows
from $h < w_0^2 \sqrt{\tau L}/L$, we see that~\eqref{boundA1} and~\eqref{boundA2} can be rephrased as
\begin{equation}\label{boundA}
\min E_h(u,\xi) \le -\frac{1}{12}\tau^2 L^3 +
Ch \sqrt{\tau L} \ln \left( w_0^{-2} \left( \frac{h L}{ \sqrt{\tau L}} \wedge 1 \right) + 1\right).
\end{equation}

\begin{remark} \label{link-to-vandeparre-etal} Our type I deformations are equivalent to
the ones discussed for heavy sheets in \cite{bib-romandrapes}. In particular, our
relation \eqref{ln} is the equivalent in our notation of equation (4) in \cite{bib-romandrapes},
giving the optimal ``length of the wrinklon.'' Since $\{l_n\}$ is a geometric series (with ratio
greater than one) we have $l_0 + \cdots + l_n \sim l_n$, so length scale of the wrinkles at
height $y$ -- call it $w(y)$ -- can be read off from \eqref{ln}: it satisfies
$|y| \sim w^2(y) \sqrt{\tau L} h^{-1}$. Rewriting this as
$\frac{w(y)}{h} \sim (\tau L)^{-1/4} (\frac{|y|}{h})^{1/2}$ we see that it is the equivalent in
our notation of equation (5) in \cite{bib-romandrapes}.
\end{remark}

\subsection{The type II deformation}\label{sectTypeII}

This construction is a modification of the previous one; besides coarsening the wrinkles it also
releases the horizontal confinement of the sheet. Let $(u,\xi)$ be the deformation from the
previous section, constructed using $N$ generations of coarsening for some $N \leq K$.
We choose a particular $n$ such that $1 \le n \le N$. We denote the $y$-coordinate of the
$k$-th building block by
\begin{equation}\nonumber
s_k := \sum_{i=1}^{k-1} l_i.
\end{equation}
For $y \in [-s_{n},0]$ the new deformation $(U,\zeta)$ coincides with $(u,\xi)$:
\begin{equation}\nonumber
\begin{aligned}
U(x,y) &:= u(x,y),\\
\zeta(x,y) &:= \xi(x,y).
\end{aligned}
\end{equation}
For $y \in [-L,-s_n]$ we define
\begin{equation}\nonumber
U_{x}(x,y) := \varphi^2\left( -\frac{s_n + y}{l_n}\right ) u_{x}(x,-s_n),\qquad
\zeta(x,y) := \varphi\left( -\frac{s_n + y}{l_n}\right ) \xi(x,-s_n),
\end{equation}
where $\varphi : [0,\infty) \to [0,1]$ is a smooth decreasing function which
satisfies $|\varphi'| \le 2$ and
\begin{equation}\nonumber
 \varphi(t) = \begin{cases} 1 & t \in (0,1/3), \\  0 & t \in (1,\infty).
              \end{cases}
\end{equation}
Then for $y \le -s_n$ we have
\begin{gather}\nonumber
U_{x,x}(x,y) + \zeta_{,x}^2(x,y)/2 =
\varphi^2\left( -\frac{s_n + y}{l_n}\right ) \cdot \left(u_{x,x}(x,-s_n) + \xi_{,x}^2(x,-s_n) /2\right) = 0.
\end{gather}

So far we did not define $U_y$ for $ \in (-L,-s_n)$. We will do this in two different ways,
thereby obtaining two different upper bounds for the energy. The first way simply sets
\begin{equation}\label{ub2}
U_y(x,y) := u_y(x,y) = f(y).
\end{equation}
For $y \in [-s_n,0]$ the excess energy is estimated by the second term in~\eqref{boundA}
(with $L$ replaced by $s_n$ in the numerator), and so we just need to estimate the excess
energy in the part of the domain $y \le -s_n$. We see
\begin{multline}\label{ub4+}
|U_{x,y} + U_{y,x} + \zeta_{,x}\zeta_{,y}|^2 =
|(2\varphi\varphi')((-s_n - y)/l_n) u_x(x,-s_n) l_n^{-1} \\
 + (\varphi\varphi')((-s_n-y)/l_n)l_n^{-1} \xi_{,x}(x,-s_n) \xi(x,-s_n)|^2 \le Cl_n^{-2},
\end{multline}
where we used
that $|\xi(\cdot,-s_n)|,|\xi_{,x}(\cdot,-s_n)|,|u_x(\cdot,-s_n)|,|\varphi'|,|\varphi| \le C$.
Using~\eqref{lnwn} we see that $|\zeta_{,y}| \le |\varphi'| |\xi| /l_n \le 2w_n / l_n \le 1$.
So \eqref{estf} implies
\begin{equation}\nonumber
|U_{y,y} + \zeta_{,y}^2/2|^2 - |f_{,y}|^2 =
f_{,y} \zeta_{,y}^2 + \zeta_{,y}^4/4 \le \tau L |\zeta_{,y}|^2 =
\tau L |\varphi'|^2l_n^{-2}\xi^2(x,-s_n) \le 4 \tau L l_n^{-2} \xi^2(x,-s_n).
\end{equation}
Since $|\xi(x,-s_n)| \le w_n$ and $l_n = w_n^2 \sqrt{\tau L} h^{-1}$, the previous
inequality implies
\begin{equation}\label{ub4}
|U_{y,y} + \zeta_{,y}^2/2|^2 - |f_{,y}|^2 \le
4 \tau L l_n^{-2} \xi^2(x,-s_n) \le
4\tau L \frac{w_n^2}{l_n} l_n^{-1} = 4 h\sqrt{\tau L} / l_n.
\end{equation}
It remains to estimate the bending energy. For $y \in [-L,s_n]$ we have by~\eqref{lnwn}
\begin{equation}\label{bendingII}
\begin{aligned}
|\zeta_{,xx}(x,y)|^2 &= \varphi^2 |\xi_{,xx}(x,-s_n)|^2 \lesssim
\omega_n^{-2} = \sqrt{\tau L}/(hl_n), \\
|\zeta_{,xy}(x,y)|^2 &= \varphi'^2 |\xi_{,x}(x,-s_n)|^2 l_n^{-2} \lesssim
l_n^{-2} \lesssim w_n^{-2} = \sqrt{\tau L}/(hl_n), \\
|\zeta_{,yy}(x,y)|^2 &= \varphi''^2 |\xi(x,-s_n)|^2 l_n^{-4} \lesssim
w_n^2 l_n^{-4} \lesssim w_n^{-2} = \sqrt{\tau L}/(hl_n),
\end{aligned}
\end{equation}
and so
\begin{equation}\label{ub2bending}
h^2 \int_{\max(-L,-s_{n+1})}^{-s_n} \int_{-1/2}^{1/2} |\nabla ^2 \zeta|^2 \dxy \lesssim
h^2 (s_n - s_{n+1}) \frac{\sqrt{\tau L}}{hl_n} = h\sqrt{\tau L}.
\end{equation}
Now we combine~\eqref{ub2bending} with~\eqref{ub4+} and \eqref{ub4} to obtain
\begin{multline}\label{ub3}
E_h(U,\zeta) \le -\frac{1}{12}\tau^2 L^3 +
C \left( h\sqrt{\tau L} \left[
\log\left( \frac{h s_n}{w_0^2 \sqrt{\tau L}} + 1\right) + 1
\right] + l_n * l_n^{-2} \right) \\
\le -\frac{1}{12}\tau^2 L^3 +
C\left( h\sqrt{\tau L} \log\left( \frac{h l_n}{w_0^2 \sqrt{\tau L}} +1\right) +  l_n^{-1}\right),
\end{multline}
where we used that $s_n$ and $l_n$ are comparable to replace $s_n$ with $l_n$ in the
logarithm, and the fact that $hl_n / (w_0^2 \sqrt{\tau L}) = 9^{n} \ge 1$.

The alternative way to choose $U_y$ will give the different upper bound
\begin{equation}\label{ub5}
E_h(U,\zeta) \le -\frac{1}{12}\tau^2 L^3 +
C \left(  h\sqrt{\tau L} \log\left( \frac{h l_n}{w_0^2 \sqrt{\tau L}}+1\right) + l_n^{-3}\right).
\end{equation}
The essential idea is to choose $U_y$ so that
\begin{equation} \label{Uy2}
U_{x,y} + U_{y,x} + \zeta_{,x}\zeta_{,y} = 0
\end{equation}
for $y \le -s_n$. Whereas our previous choice incurred a substantial energetic cost from the
term $|U_{x,y} + U_{y,x} + \zeta_{,x}\zeta_{,y}|^2$, our alternative choice makes this term
vanish, at the expense of an increase in $|U_{y,y} + \zeta_{,y}^2/2|^2$.

We can assume that $l_n \ge 1$, since otherwise~\eqref{ub5} is worse than~\eqref{ub3}.
To satisfy~\eqref{Uy2} we define
\begin{equation}\label{Uy3}
U_y(x,y) := f(y) - \int_0^x U_{x,y}(s,y) + \zeta_{,x}(s,y)\zeta_{,y}(s,y) \ud s
\end{equation}
for $x \in (-1/2,1/2)$ and $y \le -s_n$. Then
\begin{align}\nonumber
U_{y,y}(x,y) &= f_{,y}(y) -
\left(\varphi^2\left(\frac{-s_n+y}{l_n}\right)\right)_{,yy}
\left(\int_0^x u_x(s,y) + \xi_{,x}(s,y)\xi(s,y)/2 \ud s\right) \\ \nonumber
|U_{y,y}(x,y)| &\le |f_{,y}(y)| +
\left|\left(\varphi^2\left(\frac{-s_n+y}{l_n}\right)\right)_{,yy}\right|
\left|\int_0^x u_x(s,y) + \xi_{,x}(s,y)\xi(s,y)/2 \ud s\right| \\ \nonumber
&\le \tau L/2 + C|(\varphi^2)_{,yy}| \le \tau L/2 + C'l_n^{-2} \lesssim \tau L,
\end{align}
where we used that $l_n \ge 1$ and~\eqref{estf}. Therefore
\begin{multline}\nonumber
\left|U_{y,y}(x,y) + \zeta_{,y}^2(x,y)/2\right|^2 - |f_{,y}|^2 =
\left|U_{y,y}(x,y)\right|^2 - |f_{,y}|^2 + U_{y,y}(x,y)\zeta_{,y}^2(x,y) + \zeta_{,y}^4(x,y)/4 \\
\le \left|U_{y,y}(x,y)\right|^2 - |f_{,y}|^2 + C\tau L \zeta_{,y}^2.
\end{multline}
The contribution to the energy from $C\tau L \zeta_{,y}^2$ can be estimated
as in~\eqref{ub4}; it is at most  $Ch\sqrt{\tau L}$. Since $f$ is the minimizer of the bulk
energy $B$ we have that
\begin{multline}\label{ub6}
\left(\iint_\Omega \left|U_{y,y}\right|^2 + \tau U_y \dxy \right)-
\left( \iint_\Omega |f_{,y}|^2 + \tau f \dx \dy \right) \\
= B(U_y) - B(f) = \frac{1}{2} \left(D^2B(f) (U_y - f), U_y - f\right) =
||U_{y,y} - f_{,y}||_{L^2(\Omega)}^2 \\
\lesssim \iint_\Omega \left| \left( \varphi^2 \left( -\frac{s_n + y}{l_n}\right ) \right)_{,yy} \right|^2 \dxy
\lesssim l_n * l_n^{-4} = l_n^{-3}.
\end{multline}
The estimate for the bending energy~\eqref{ub2bending} remains valid, and so the combination
of~\eqref{ub6} with~\eqref{Uy2} implies~\eqref{ub5}.

\subsection{The type III deformation} \label{sectTypeIII}

The construction in this section is closely related to the discussion at the end of
Section \ref{sectTypeI}, which considered the consequences of ``propagating the
deformation prescribed at $\Gamma_T$'' (see \eqref{just-propagate}).
Here we do something similar, but we release the horizontal confinement by a mechanism similar to that of
Section \ref{sectTypeII}.

Choose any $l \in (w_0,L)$, and consider
\begin{equation}\nonumber
U_{x}(x,y) := \varphi^2\left( -y/l\right) u_{x}(x,0),\qquad U_{y}(x,y) := f(y), \qquad
\zeta(x,y) := \varphi\left( -y/l\right ) \xi(x,0).
\end{equation}
Then $|U_{x,x}(x,y) + \zeta_{,x}^2(x,y)/2|=0$ for all $(x,y)\in \Omega$, and
\begin{gather}\nonumber
\iint_\Omega |U_{x,y} + U_{y,x} + \zeta_{,x} \zeta_{,y}|^2 \dxy
\lesssim l * l^{-2} = l^{-1}, \\ \nonumber
\iint_\Omega |U_{y,y} + \zeta_{,y}^2/2|^2 - |f_{,y}|^2 \dxy
\lesssim \tau L l^{-1} \int_{-1/2}^{1/2} |\xi(x,0)|^2 \dx \lesssim w_0^2 \tau L l^{-1}.
\end{gather}
Since $l \ge w_0$, a calculation similar to~\eqref{bendingII} shows that
$h^2 \iint |\nabla^2 \zeta|^2 \lesssim h^2 w_0^{-2} l$. Combining these estimates gives
\begin{equation}\label{ubspecial1}
E_h(U,\zeta) \le -\frac{1}{12}\tau^2 L^3 + C\left( h^2 w_0^{-2} l + l^{-1} + w_0^2 \tau L l^{-1}\right).
\end{equation}

In Section \ref{sectTypeII} we considered two different ways of extending $U_y$. The preceding calculation
is like the first, but we can also consider the second. Using~\eqref{Uy3} to define $U_y$ and proceeding
as above we find the estimate
\begin{equation}\label{ubspecial2}
E_h(U,\zeta) \le -\frac{1}{12}\tau^2 L^3 + C\left( h^2 w_0^{-2} l + l^{-3} + w_0^2 \tau L l^{-1}\right).
\end{equation}

Since the right-hand side of~\eqref{ubspecial1} as a function of $l \in (0,L)$ attains its minimum
for $l \ge w_0$, we immediately observe that~\eqref{ubspecial1} holds for
all $l \in (0,L)$. Finally, for the same reason~\eqref{ubspecial2} holds for
all $l \in (0,L)$ as well.

Taken together, the upper bounds \eqref{boundA}, \eqref{ub3}, \eqref{ub5}, \eqref{ubspecial1}, and
\eqref{ubspecial2} establish \eqref{ub1}. Thus we have proved the upper bound half of Theorem \ref{thm1}.


\section{The lower bound}\label{sec:lb}

This section proves our lower bound for the minimum energy. As shown in Section~\ref{sec:rescaling},
it suffices to consider $\Delta =1$ and $W=1$. Our task is therefore to show that if~\eqref{tauL}
and~\eqref{wgeh} hold and $(u,\xi)$ satisfies the prescribed boundary condition~\eqref{bdry}, then
the excess energy satisfies a lower bound of the form
\begin{multline} \nonumber
\delta := E_h(u,\xi) - B(f) \ge C_{LB}\min \left(
h \sqrt{\tau L} \log \left( w_0^{-2} \left( \frac{h L}{\sqrt{\tau L}} \wedge 1\right) + 1 \right), \right. \\
\left. \min_{l \in (0,L)} \left\{ h \sqrt{\tau L} \log \left( \frac{h l}{w_0^{2}\sqrt{\tau L}} + 1 \right) +
w_0^2 \tau L l^{-1} + \min \left ( l^{-1}, l^{-3} \right ) \right\} \right).
\end{multline}
Here $f$ is the minimizer of the bulk energy functional $B$, and $C_{LB} > 0$ is a
(sufficiently small) constant that does not depend on the parameters of our problem
(several smallness conditions on $C_{LB}$ will emerge in the course of the proof). We will
argue by contradiction; in fact, our strategy is to assume that
\begin{multline}\label{deltaless}
\delta \le C_{LB} \min\left(
h \sqrt{\tau L} \log \left( \frac{h L}{w_0^2 \sqrt{\tau L}} + 1 \right), \right. \\
\left. \min_{l \in (0,L)} \left\{ h \sqrt{\tau L} \log \left( \frac{h l}{w_0^2 \sqrt{\tau L}} + 1 \right) +
w_0^2 \tau L l^{-1} + \min \left ( l^{-1}, l^{-3} \right ) \right\} \right)
\end{multline}
and to prove using this smallness condition on $\delta$ that our lower bound (the opposite inequality)
must hold.

\begin{remark}\label{smoothness}
We can assume without loss of generality that the deformation $(u,\xi)$ is smooth, since mollification has only
a small effect on the energy. (It is important here that our goal is the scaling law, not the optimal value
of the prefactor $C_{LB}$.)
\end{remark}

\subsection{The idea of the proof} \label{subsec:7-1}

Before beginning the proof in full detail let us sketch the main steps.

\begin{itemize}
\item
Since the sheet is stretched vertically, it prefers not to change its out-of-plane
displacement $\xi$. The situation is similar to a stretched rubber band, whose
preferred configuration is the straight line joining its endpoints (and for which deviation
from this configuration costs additional elastic energy). We also know that when $h=0$
it is optimal to have $u_y = f$, and we expect similar behavior for $h > 0$ (modulo
small adjustments due to wrinkling). We make these ideas quantitative in Lemma~\ref{lm-xi1}.

\item
In some cases we expect the sheet to spread laterally, releasing the horizontal confinement
prescribed at $\Gamma_T$. Mathematically speaking, spreading entails decreasing
the value of $|u_x|$, which is of order $1$ at $y=0$. If the value of $u_x$ is decreased
significantly over a length $\l$, we expect $u_{x,y}$ to be of order $\l^{-1}$.
The term $u_{x,y}$ appears in the energy in $|u_{x,y} + u_{y,x} + \xi_{,x}\xi_{,y}|$.
So we are left with two alternatives: either $u_{y,x}$ and $\xi_{,x}\xi_{,y}$ are negligible, and
$u_{x,y}$ being of order $\l^{-1}$ over a domain of size $\l$ makes the excess energy at least
of order $\l^{-1}$; or else one of the terms $u_{y,x}$ or $\xi_{,x}\xi_{,y}$ is of order $\l^{-1}$.
In the second case other terms in the energy come into play. In Lemma~\ref{dichotomy} we show
that if the sheet releases a non-trivial part of its confinement over length $\l$, then the
excess energy satisfies $\delta \gtrsim \min\left(\l^{-1},\l^{-3}\right)$.

The result just sketched shows that if $\delta \lesssim \min(\l^{-1},\l^{-3})$ then for
$y \in (-\l,0)$ the sheet is ``fairly'' confined. Though the rigorous result (Lemma~\ref{dichotomy})
states this fact in an integral form, a heuristic version of its conclusion is that
$u_x(-1/2,y) \approx 1/2$ and $u_x(1/2,y) \approx -1/2$ for $y \in (-\l,0)$. It follows that
\begin{equation}\nonumber
1 \approx u_x(-1/2,y) - u_x(1/2,y) \le \int_{-1/2}^{1/2} |u_{x,x}(x,y)| \dx \le
\int_{-1/2}^{1/2} |u_{x,x} + \xi_{,x}^2/2| \dx + \int_{-1/2}^{1/2} \xi_{,x}^2/2 \dx,
\end{equation}
so at least one of the terms on the right-hand side must be of order $1$ (see Lemma~\ref{dichotomy3}
for the details).

\item
The type I deformation considered in Section~\ref{sec:ub} involves a coarsening cascade of
wrinkles. It is preferred over our other constructions in a certain (rather large) region
of parameter space. Qualitatively, coarsening is favorable because coarser wrinkles have less
bending energy. Lemma~\ref{lmK} makes this quantitative, by showing that the $L^2$ norm of
the out-of-plane displacement $\xi$ cannot stay uniformly small.

\item
Continuing the discussion of coarsening: the boundary condition at the top makes $\xi(\cdot,0)$
small in $L^2$, while it follows from the previous step that $\xi(\cdot,y_0)$ is large for some
$y_0 \in (-\l,0)$ (if the confinement is not released, and if the excess energy is small). In
Lemma~\ref{mainlemma} we show that there is an energetic cost associated with changing the scale of
the wrinkling; in fact, each doubling of $||\xi(\cdot,y)||_{L^2}$ costs an energy of order
$Ch\sqrt{\tau L}$. This lemma is, roughly speaking, a lower-bound analogue of the
passage from~\eqref{sumenergy} to~\eqref{boundA0}. Its proof builds on an argument from \cite{bib-kohnnguyen-raft}.

\item
The preceding bullets sketch the proof of the lower bound in the regime where self-similar
coarsening is desirable. We also prove lower bounds associated with other
regimes (for example when a type III deformation is best). If the bending term
$\|\xi_{,xx}(\cdot,y)\|_{L^2}$ does not decrease much, we obtain the lower bound
$h^2 w_0^{-2} L$ by simply integrating the bending energy over the domain.
Otherwise we can assume that for some $y_0 \in (-L/2,0)$ the bending term $\int \xi_{,xx}^2(x,y_0) \dx$ is
much smaller than at the beginning ($y=0$). In this case we consider $\eta(x) := \xi(x,y_0) - \xi(x,0)$. We show
that $||\eta_{,x}||_{L^2}$ is of order $1$ while $||\eta_{,xx}||_{L^2}$ is bounded by $Cw_0^{-1}$. These yield
a lower bound for $||\eta||_{L^2}$, obtained as a consequence of the interpolation inequality
$||\eta_{,x}||_{L^2}^2 \le C_{int} \left( ||\eta||_{L^2} ||\eta_{,xx}||_{L^2} +  ||\eta||_{L^2}^2\right)$.
But $||\eta||_{L^2}$ is controlled by $\iint_\Omega \xi_{,y}^2$, which in turn is controlled by the excess
energy. We thus obtain another lower bound on the excess energy.
\end{itemize}

\subsection{The dichotomy}

We start with a lemma relating the excess energy to the out-of-plane displacement $\xi$ and the strain term
$u_{y,y} + \xi_{,y}^2 / 2 - f_{,y}$.

\begin{lemma}\label{lm-xi1}
Let $(u,\xi)$ be any deformation satisfying~\eqref{bdry}, and recall our definition of it excess energy:
\begin{equation} \nonumber
\delta := E_h(u,\xi) - B(f),
\end{equation}
where $f$ is the minimizer of the bulk energy $B$ (see Section~\ref{sec:bulk}).
Under the assumptions of Theorem~\ref{thm1} the excess energy is nonnegative, and
\begin{equation}\label{xiy}
\int_{\Omega} |u_{y,y} + \xi_{,y}^2/2 - f_{,y}|^2 \dxy \le \delta , \qquad
\iint_{\Omega/2} \xi_{,y}^2 \dxy \le 8\delta (\tau L)^{-1},
\end{equation}
\begin{equation}\label{otherterms}
\iint_\Omega \frac{1}{2}|u_{x,y} + u_{y,x} + \xi_{,x}\xi_{,y}|^2 + |u_{x,x} + \xi_{,x}^2/2|^2 +
h^2 |\nabla^2 \xi|^2 \dxy \le \delta,
\end{equation}
where $\Omega / 2 := (-1/2,1/2) \times (-L/2,0)$ is the upper half of the domain.
\end{lemma}

\begin{proof}
We know that
\begin{equation}\label{Eh}
\begin{aligned}
E_h ( u,\xi )  &= \iint_\Omega |u_{x,x} + \xi_{,x}^2/2|^2 + \oh|u_{x,y}+u_{y,x}+\xi_{,x}\xi_{,y}|^2 +
|u_{y,y} + \xi_{,y}^2/2|^2 + h^2 | \nabla^2 \xi |^2 \dxy \\
&\quad + \tau \iint_\Omega u_y \dxy \\
&\ge \iint_\Omega |u_{y,y} + \xi_{,y}^2/2|^2 + \tau u_y \dxy =
B(u_y) + \iint_\Omega u_{y,y}\xi_{y}^2 + \xi_{,y}^4/4 \dxy.
\end{aligned}
\end{equation}
Since $f$ is the minimizer of $B$ (in particular it is a critical point of $B$), using the previous
relation we get
\begin{equation}\label{delta}
\begin{aligned}
\delta & =E_h(u,\xi) - B(f) \ge
\frac{1}{2} \left( D^2B \cdot (u_y - f), u_y-f \right) + \iint_\Omega u_{y,y}\xi_{,y}^2 + \xi_{,y}^4/4 \dxy \\
&= \iint_\Omega (u_{y,y} - f_{,y})^2 + u_{y,y}\xi_{,y}^2 + \xi_{,y}^4/4 \dxy =
\iint_\Omega (u_{y,y} - f_{,y} + \xi_{,y}^2/2)^2 + f_{,y}\xi_{,y}^2 \dxy \ge 0.
\end{aligned}
\end{equation}
We see that if $E_h(u,\xi)$ is close to $B(f)$ then $u_{y,y} + \xi_{,y}^2/2$ is close to
$f_{,y}$ and $\xi_{,y}^2$ is bounded. More quantitatively, we have
\begin{align*}
\int_{\Omega} | u_{y,y} + \xi_{,y}^2/2 - f_{,y} |^2 \dxy &\le \delta , \\
\frac{\tau L}{8} \iint_{\Omega/2} \xi_{,y}^2 \dxy \le \iint_{\Omega} f_{,y}\xi_{,y}^2/2 &\le \delta,
\end{align*}
where we have used that $f_{,y} \geq 0$ for $y \in (-L,0)$ and $f_{,y}(y) \ge \tau L/4$ for $y \in (-L/2,0)$
(see~\eqref{estf}). Finally, we observe that in~\eqref{delta} we have used only one of the four
non-negative terms in $E_h$, so \eqref{otherterms} follows.
\end{proof}

We must consider the possibility that the sheet spreads laterally, releasing its confinement (at least in part).
In the following lemma we estimate the energy needed to do that.



\begin{lemma}
\label{dichotomy}
Under the assumptions of Lemma~\ref{lm-xi1} there exist $-1/2 < x_0 < -7/16$ and $-3/8 < x_1 < -5/16$,
and a universal constant $C_0 > 0$, such that either
\begin{enumerate}
\item[i)] there exists $l_1 \in (0,L/2)$ such that
\begin{equation} \nonumber \delta \max (l_1, l_1^3) \ge C_0 \end{equation}
and
\begin{equation}\label{kappa}
\left| \int_{x_0}^{x_1} u_x(x,y) - u_x(x,0) \dx \right| \le 1/512
\end{equation}
for all $y \in (-l_1,0)$;
\item[ii)] or~\eqref{kappa} holds for all $y \in (-L/2,0)$ (we set $l_1:= L/2$ in this case).
\end{enumerate}
\end{lemma}

Before proving the lemma let us sketch the main ingredients in its proof. First, we use the bounds
from Lemma~\ref{lm-xi1} to show that $u_{y}$ is in average close to $f$. In particular,
$\int |u_y(x,y) - f(y)| \dy$ is small for generic $x$, and we can choose such generic
$x_0$ and $x_1$ near $x=-1/2$ such that the size of the interval $(x_0,x_1)$ is of order $1$.

To get \eqref{kappa} it is enough to estimate $u_{x,y}$ since
$\int_{x_0}^{x_1} u_x(x,y_0) - u_x(x,0) \dx = \int_{x_0}^{x_1} \int_{y_0}^0 u_{x,y} \dxy$. We write
$u_{x,y} = (u_{x,y} + u_{y,x} + \xi_{,x}\xi_{,y}) - u_{y,x} - \xi_{,x}\xi_{,y}$. The first term on
the right-hand side can be estimated using the excess energy $\delta$ (see \eqref{otherterms}).
We integrate the second term in $x$ and use that by the choice of $x_0,x_1$ both $u_y(x_0,\cdot)$
and $u_y(x_1,\cdot)$ are close to $f$; it follows by the triangle inequality that $u_y(x_0,\cdot)$ is
close to $u_y(x_1,\cdot)$.

It remains to estimate $\xi_{,x}\xi_{,y}$. First, we write $\xi(x,y) = \int \xi_{,y}(x,\cdot) + \xi(x,0)$
and use Lemma~\ref{lm-xi1}. Then using interpolation with $\xi_{,xx}$ we obtain enough control
over $\xi_{,x}$. In the end we obtain an estimate for
$\left| \int_{x_0}^{x_1} u_x(x,y) - u_x(x,0) \dx \right|$ in terms of a function of
$\delta \max(l_1,l_1^3)$, which concludes the proof.

In proving Lemma~\ref{dichotomy}, and in many other places in this section, we will use the basic
interpolation inequality
\begin{equation}\label{interpolation}
||\varphi'||_{L^2(I)}^2 \le
C_{int}\left( ||\varphi||_{L^2(I)}||\varphi''||_{L^2(I)} + \frac{1}{|I|^{2}} ||\varphi||_{L^2(I)}^2\right),
\end{equation}
which holds for any interval $I \subset \mathbb{R}$ and any $\varphi \in W^{2,2}(I)$
with a fixed universal constant $C_{int}$.

\begin{proof}[Proof of Lemma \ref{dichotomy}]
To start we need to choose $x_0$ and $x_1$ generically. By~\eqref{xiy} we have
\begin{gather*}
\iint_\Omega |u_{y,y} + \xi_{,y}^2/2 - f_{,y}|^2 \dx \dy \le \delta,\\
\iint_{\Omega/2} \xi_{,y}^2 /2 \dx \dy \le 4 \delta (\tau L)^{-1},
\end{gather*}
so we can choose $-1/2 < x_0 < -7/16$ and $-3/8 < x_1 < -5/16$ such that
\begin{equation}\label{x0x1}
\int_{-L}^0 |u_{y,y}(x_i,y) + \xi_{,y}^2(x_i,y)/2 - f_{,y}(y)|^2 \dy \lesssim \delta, \qquad
\int_{-L/2}^0 \xi_{,y}^2(x_i,y) /2 \dy \lesssim \delta (\tau L)^{-1}, \qquad i = 0,1.
\end{equation}

Now suppose \eqref{kappa} is not true for all $y \in (-L/2,0)$; then there exists
a smallest $l_1 \in (0,L/2)$ such that
\begin{equation} \nonumber
\left| \int_{x_0}^{x_1} u_x(x,-l_1) - u_x(x,0) \dx \right| \ge 1/512 .
\end{equation}
By the boundary condition~\eqref{bdry} we have $u_y(\cdot,0) = 0$. So~\eqref{x0x1} implies that
for any $y \in (-l_1,0)$ and $i=0,1$:
\begin{equation}\nonumber 
\begin{aligned}
|u_y(x_i,y) - f(y)| &\le \int_{y}^0 |u_{y,y}(x_i,y) - f_{,y}(y)| \dy \\
&\le \int_{y}^0 |u_{y,y}(x_i,y) + \xi_{,y}^2(x_i,y)/2 - f_{,y}(y)| \dy +
\int_{y}^0 \xi_{,y}^2(x_i,y)/2 \dy \\
&\overset{\textrm{H\"older}}{\le}
l_1^{1/2} \left(\int_{-L}^0 |u_{y,y}(x_i,y) + \xi_{,y}^2(x_i,y)/2 - f_{,y}(y)|^2 \dy\right)^{1/2} +
\int_{y}^0 \xi_{,y}^2(x_i,y)/2 \dy \\
&\overset{\eqref{x0x1}}{\lesssim} \delta^{1/2} l_1^{1/2} + \delta (\tau L)^{-1}.
\end{aligned}
\end{equation}
We integrate the previous relation in $y$ and use $\tau L \ge 4$ (see \eqref{tauL}) to obtain
\begin{equation}\label{uy3}
 \int_{-l_1}^0 |u_y(x_i,y) - f(y)| \dy \lesssim \delta^{1/2}l_1^{3/2} + \delta l_1, \quad i=0,1.
\end{equation}
We define $\Omega' := [x_0,x_1] \times [-l_1,0]$. Then
\begin{multline}\label{ux}
\left| \int_{x_0}^{x_1} u_x(x,-l_1) - u_x(x,0) \dx \right| =
\left| \iint_{\Omega'} u_{x,y} \dxy \right| \\
\le \iint_{\Omega'} |u_{x,y} + u_{y,x} + \xi_{,x} \xi_{,y}| \dxy  +
\left| \iint_{\Omega'} u_{y,x} \dxy\right| + \left| \iint_{\Omega'} \xi_{,x} \xi_{,y} \dxy\right|.
\end{multline}
By~\eqref{otherterms} the first integral on the right-hand side satisfies
\begin{equation}\label{est1}
\iint_{\Omega'} |u_{x,y} + u_{y,x} + \xi_{,x} \xi_{,y}| \dxy \overset{\textrm{H\"older}}{\le}
|\Omega'|^{1/2} \left( \iint_{\Omega'} |u_{x,y} + u_{y,x} + \xi_{,x} \xi_{,y}|^2 \dxy \right)^{1/2} \lesssim
\delta^{1/2} l_1^{1/2}.
\end{equation}
For the second integral we have
\begin{multline}\label{est2}
\left| \iint_{\Omega'} u_{y,x} \dxy\right|  = \left| \int_{-l_1}^0 u_{y}(x_1,y) - u_y(x_0,y) \dy \right| \\
\le \int_{-l_1}^0 \left| u_{y}(x_1,y) - f(y)\right| \dy + \int_{-l_1}^0 \left| u_{y}(x_0,y) - f(y)\right| \dy
\lesssim \delta^{1/2} l_1^{3/2} + \delta l_1,
\end{multline}
where the last inequality follows from~\eqref{uy3}. It remains to estimate the last term in~\eqref{ux}:
\begin{equation} \nonumber
\left| \iint_{\Omega'} \xi_{,x}\xi_{,y} \dxy \right|.
\end{equation}
Since we already have bound on $\xi_{,y}$ (see~\eqref{xiy}) it remains to estimate $\xi_{,x}$.
To do this, we first show a bound on $\xi$ by integrating $\xi_{,y}$, then we use interpolation
with $\xi_{,xx}$. For any $y \in [-l_1,0]$ we have
\begin{equation}\nonumber
\begin{aligned}
\int_{x_0}^{x_1} |\xi(x,y)|^2 \dx &=
\int_{x_0}^{x_1} \left| \int_{y}^0 \xi_{,y}(x,t) \ud t + \xi(x,0)\right|^2 \dx \\
&\le 2\int_{x_0}^{x_1} \left| \int_{y}^0 \xi_{,y}(x,t) \ud t \right|^2 \dx +
2\int_{x_0}^{x_1} |\xi(x,0)|^2 \dx \\
&\le 2|y|\int_{x_0}^{x_1} \int_{y}^0 |\xi_{,y}(x,t)|^2 \ud t \dx +
2\int_{x_0}^{x_1} |\xi(x,0)|^2 \dx \lesssim
\delta l_1 (\tau L)^{-1} + w_0^2,
\end{aligned}
\end{equation}
where the last inequality follows from~\eqref{bdry} and~\eqref{xiy}. We use the interpolation
inequality~\eqref{interpolation} for $\xi$ to get
\begin{align*}
\left(\int_{x_0}^{x_1} |\xi_{,x}(x,y)|^2 \dx\right)^2 &\lesssim
\left(\int_{x_0}^{x_1} |\xi(x,y)|^2 \dx\right) \left(\int_{x_0}^{x_1} |\xi_{,xx}(x,y)|^2 \dx\right) \\
&\quad + |x_1-x_0|^{-4} \left(\int_{x_0}^{x_1} |\xi(x,y)|^2 \dx\right)^2 \\
&\lesssim (\delta l_1 (\tau L)^{-1} + w_0^2) \int_{x_0}^{x_1} |\xi_{,xx}(x,y)|^2 \dx +
(\delta l_1 (\tau L)^{-1} + w_0^2)^2.
\end{align*}
After integration in $y$ we obtain
\begin{multline} \nonumber
\left( \iint_{\Omega'} |\xi_{,x}(x,y)|^2 \dx \dy\right)^2 \le
l_1 \int_{y_0}^0 \left(\int_{x_0}^{x_1} |\xi_{,x}(x,y)|^2 \dx\right)^2 \\
\lesssim l_1 (\delta l_1 (\tau L)^{-1} + w_0^2)
\iint_{\Omega'} |\xi_{,xx}(x,y)|^2 \dxy + l_1 (\delta l_1 (\tau L)^{-1} + w_0^2)^2 \\
\le l_1 (\delta l_1 (\tau L)^{-1} + w_0^2) \delta h^{-2} + l_1 (\delta l_1 (\tau L)^{-1} + w_0^2)^2,
\end{multline}
where the last inequality follows from~\eqref{otherterms}. This combined with~\eqref{xiy} implies
\begin{multline}\label{est3}
\left| \iint_{\Omega'} \xi_{,x}\xi_{,y} \dxy\right| \le
||\xi_{,x}||_{L^2(\Omega')} ||\xi_{,y}||_{L^2(\Omega')} \\
\lesssim \sqrt[4]{ \left( \delta l_1^2 (\tau L)^{-1} + w_0^2 l_1\right) \delta h^{-2} +
l_1 (\delta l_1 (\tau L)^{-1} + w_0^2)^2} \sqrt{ \delta (\tau L)^{-1} }.
\end{multline}
Combination of~\eqref{est1}, \eqref{est2}, and \eqref{est3} with~\eqref{ux} gives
\begin{multline}\label{uxbdry}
\left| \int_{x_0}^{x_1} u_x(x,-l_1) - u_x(x,0) \dx \right| \\
\lesssim \left( \delta^{1/2} l_1^{1/2} + \delta^{1/2} l_1^{3/2} + \delta l_1  \right) \\
+ \delta l_1^{1/2} h^{-1/2} (\tau L)^{-3/4} + w_0^{1/2} \delta^{3/4} l_1^{1/4} h^{-1/2} (\tau L)^{-1/2} +
\delta l_1^{3/4} (\tau L)^{-1} + w_0 \delta^{1/2} l_1^{1/4} (\tau L)^{-1/2}.
\end{multline}

We want to estimate the four terms on the last line of~\eqref{uxbdry} in terms of $\delta l_1$ and $\delta l_1^3$.
By~\eqref{deltaless} we know
\begin{equation}\label{pp1}
\delta \le C_{LB}
\left( h \sqrt{\tau L} \log \left( \frac{h l}{w_0^2 \sqrt{\tau L}} + 1 \right) + w_0^2 \tau L l^{-1} + l^{-1}\right)
\end{equation}
for any $l \in (0,L)$. It is easy to observe that for $l \ge L$ the relation~\eqref{pp1} holds as well since
the right-hand side of~\eqref{pp1} is then larger than
$C_{LB} h \sqrt{\tau L} \log \left( \frac{h L}{w_0^2 \sqrt{\tau L}} + 1 \right)$, which is in turn larger
than $\delta$ by~\eqref{deltaless}. We now estimate the last line of~\eqref{uxbdry} by making three separate
applications of~\eqref{pp1}:
\begin{enumerate}
\item[(1)]
Taking $l := h^{-3/5} (\tau L)^{1/10}$ in~\eqref{pp1} and using that $\log(1+t) \le 2t^{1/4}$ for $t \ge 0$
we obtain
\begin{equation}\label{12-pom1}
\delta \le C_{LB} \left(
2h^{3/5} \left(\frac{h}{w_0}\right)^{1/2} (\tau L)^{2/5} + w_0^2 h^{3/5} (\tau L)^{9/10} + h^{3/5}(\tau L)^{-1/10}
\right).
\end{equation}
Since the first term in the last line of~\eqref{uxbdry} can be written as
\begin{equation} \nonumber
\delta l_1^{1/2} h^{-1/2} (\tau L)^{-3/4} =
(\delta l_1^3)^{1/6} \left( \delta h^{-3/5} (\tau L)^{-9/10} \right)^{5/6},
\end{equation}
we can control it by estimating $\delta h^{-3/5} (\tau L)^{-9/10} $.
Using $h/w_0 \le 1$ (see~\eqref{wgeh}), $w_0 \le 1$, and $\tau L \ge 4$ (see \eqref{tauL}), we obtain
from~\eqref{12-pom1} that
\begin{equation} \nonumber
\delta h^{-3/5} (\tau L)^{-9/10} \le C_{LB}\left( 2(\tau L)^{-5/10} + w_0^2 + (\tau L)^{-1} \right) \lesssim 1,
\end{equation}
which implies
\begin{equation}\label{12-ing3}
\delta l_1^{1/2} h^{-1/2} (\tau L)^{-3/4} \lesssim (\delta l_1^3)^{1/6}.
\end{equation}
(The implicit constant in~\eqref{12-ing3} involves a positive power of $C_{LB}$. But we may (and shall)
assume that $C_{LB} \leq 1$; then~\eqref{12-ing3} holds with an implicit constant that is independent of $C_{LB}$.)

\item[(2)]
For the next term we use~\eqref{pp1} with $l := h^{-1}$. Using $\log(1+t) \le t^{1/2}$ for $t \ge 0$
it follows from~\eqref{pp1} that
\begin{equation} \nonumber
\delta \le C_{LB}\left( h w_0^{-1} (\tau L)^{1/4} + w_0^2 \tau L h + h\right),
 \end{equation}
so we have
\begin{multline}\label{12-ing4}
w_0^{1/2} \delta^{3/4} l_1^{1/4} h^{-1/2} (\tau L)^{-1/2} \le
(\delta l_1)^{1/4} \left( \delta w_0 h^{-1} (\tau L)^{-1} \right)^{1/2} \\
\le C_{LB}^{1/2} (\delta l_1)^{1/4} \left( (\tau L)^{-3/4} + w_0^3 + w_0 (\tau L)^{-1} \right)^{1/2}
\lesssim (\delta l_1)^{1/4}.
\end{multline}

\item[(3)]
The last two terms are fairly easy to estimate. First, we observe that $\delta (\tau L)^{-1} \lesssim 1$.
Indeed, choosing $l := 1$ in~\eqref{pp1} gives
\begin{equation}\nonumber
\delta \lesssim \left( h\sqrt{\tau L} \log\left( \frac{h}{w_0^2 \sqrt{\tau L}} + 1\right) +
w_0^2 \tau L + 1\right) \le 2 + w_0^2 \tau L,
\end{equation}
where we used $\log(1+t) \le t$ and $h \le w_0$. Since $\tau L \ge 4$ and $w_0 \le \cw$
(see \eqref{tauL} and~\eqref{wgeh}), we have $\delta (\tau L)^{-1} \lesssim 1 + w_0^2 \lesssim 1$. Using this
estimate for $\delta$, the last two terms are estimated as follows:
\begin{gather*}
\delta l_1^{3/4} (\tau L)^{-1} =
(\delta l_1)^{3/4} (\delta (\tau L)^{-1})^{1/4} (\tau L)^{-3/4} \lesssim (\delta l_1)^{3/4},\\
w_0 \delta^{1/2} l_1^{1/4} (\tau L)^{-1/2} =
w_0 (\delta l_1)^{1/4} (\delta (\tau L)^{-1})^{1/4} (\tau L)^{-1/4} \lesssim (\delta l_1)^{1/4},
\end{gather*}
where we have used~\eqref{tauL} and~\eqref{wgeh}.
\end{enumerate}
Combining~\eqref{12-ing3} and~\eqref{12-ing4} with~\eqref{uxbdry} and the last estimate we get that
\begin{equation} \nonumber
\left| \int_{x_0}^{x_1} u_x(x,-l_1) - u_x(x,0) \dx \right| \lesssim
(\delta l_1)^{1/2} + (\delta l_1^3)^{1/2} + \delta l_1+ (\delta l_1^3)^{1/6} +
(\delta l_1)^{1/4} + (\delta l_1)^{3/4}.
\end{equation}
By our choice of $l_1$ we see that
\begin{equation} \nonumber
1/512 \le \left| \int_{x_0}^{x_1} u_x(x,-l_1) - u_x(x,0) \dx \right| \lesssim F(\delta \max (l_1,l_1^3)),
\end{equation}
where $F(t) = t^{1/6} + t^{1/4} + t^{1/2} + t^{3/4} + t$. It follows that
\begin{equation} \nonumber
\delta \max (l_1,l_1^3) \ge C_0
\end{equation}
for a universal constant $C_0 > 0$.
\end{proof}


\begin{remark}\label{dichotomy2}
By symmetry we can find $5/16 < x_2 < 3/8$ and $7/16 < x_3 < 1/2$ for which either
\begin{enumerate}
 \item[i)] there exists $l_2 \in (0,L/2]$ such that
\begin{equation} \nonumber \delta \max (l_2, l_2^3) \ge C_0 \end{equation}
and
\begin{equation}\label{kappa2}
 \left| \int_{x_2}^{x_3} u_x(x,y) - u_x(x,0) \dx \right| \le 1/512,
\end{equation}
holds for all $y \in (-l_2,0)$;
\item[ii)] or~\eqref{kappa2} holds for all $y \in (-L/2,0)$ (we set $l_2 := L/2$ in this case).
\end{enumerate}
\end{remark}
We use this remark together with Lemma~\ref{dichotomy} to get
\begin{cor}\label{dich}
There exist $-1/2 < x_0 < -7/16$, $-3/8 < x_1 < -5/16$,  $5/16 < x_2 < 3/8$,
$ 7/16 < x_3 < 1/2$, and a universal constant $C_0$, such that either
\begin{enumerate}
\item[i)] for any $y \in (-L/2,0)$
\begin{equation}\label{kappa3}
\begin{aligned}
\left| \int_{x_0}^{x_1} u_x(x,y) - u_x(x,0) \dx \right| \le 1/512,\\
\left| \int_{x_2}^{x_3} u_x(x,y) - u_x(x,0) \dx \right| \le 1/512,
\end{aligned}
\end{equation}
\item[ii)] or there exists $\l \in (0,L/2]$ such that
\begin{equation} \nonumber \delta \max (\l, \l^3) \ge C_0 \end{equation}
and~\eqref{kappa3} holds for all $y \in (-\l,0)$.
\end{enumerate}
When the first alternative holds, we take the convention that $\l = L/2$.
\end{cor}

\begin{proof}
We obtain this result by combining Lemma~\ref{dichotomy} with Remark~\ref{dichotomy2}, choosing
$\l := \min(l_1,l_2)$.
\end{proof}

The next step in the proof of the lower bound is to show that $\int \xi_{,x}^2(x,y_0) \dx$ and
$\int |u_{x,x}(x,y_0) + \xi_{,x}^2(x,y_0)/2| \dx$ can not be simultaneously small for any $y_0 \in (-\l,0)$.

\begin{lemma}\label{dichotomy3}
For any $y_0 \in (-\l,0)$ we have either
\begin{equation} \nonumber
\int_{x_0}^{x_3} \xi_{,x}^2(x,y_0) \dx \ge 1/4
\end{equation}
or
\begin{equation} \nonumber
\int_{x_0}^{x_3} \left|u_{x,x}(x,y_0) + \xi_{,x}^2(x,y_0)/2\right| \dx \ge 1/16.
\end{equation}
\end{lemma}

\begin{proof}
Let $y_0 \in (-\l,0)$ be fixed and
$x_l \in (x_0,x_1), x_r \in (x_2,x_3)$. We have
\begin{align} \nonumber
u_x(x_r,0) - u_x(x_l,0) &= u_x(x_r,y_0) - u_x(x_l,y_0) +
\left[ (u_x(x_r,0) - u_x(x_r,y_0) ) - ( u_x(x_l,0) - u_x(x_l,y_0)) \right] \\  \nonumber
 &= \int_{x_l}^{x_r} u_{x,x}(x,y_0) \dx +
 \left[ (u_x(x_r,0) - u_x(x_r,y_0) ) - ( u_x(x_l,0) - u_x(x_l,y_0)) \right] \\ \nonumber
&= \int_{x_l}^{x_r} u_{x,x}(x,y_0) + \frac{1}{2}\xi_{,x}^2(x,y_0) \dx -
\frac{1}{2}\int_{x_l}^{x_r} \xi_{,x}^2(x,y_0) \dx + [\dots].
\end{align}
Therefore
\begin{align} \nonumber
\frac{1}{2}\int_{x_0}^{x_3} \xi_{,x}^2(x,y_0) \dx &\ge
\frac{1}{2}\int_{x_l}^{x_r} \xi_{,x}^2(x,y_0) \dx \\  \nonumber
&\ge u_x(x_l,0) - u_x(x_r,0) + \int_{x_l}^{x_r} u_{x,x}(x,y_0) +
\frac{1}{2}\xi_{,x}^2(x,y_0) \dx + [\dots] \\ \nonumber
&\ge u_x(x_l,0) - u_x(x_r,0) - \int_{x_l}^{x_r} \left| u_{x,x}(x,y_0) +
\frac{1}{2}\xi_{,x}^2(x,y_0)\right| \dx + [\dots].
\end{align}
We integrate the previous relation with respect to $x_l$ and $x_r$ to obtain
\begin{multline}\label{xix}
\frac{1}{2}\int_{x_0}^{x_3} \xi_{,x}^2(x,y_0) \dx \ge
\fint_{x_0}^{x_1} \fint_{x_2}^{x_3} u_x(x_l,0) - u_x(x_r,0) \ud x_r \ud x_l -
\int_{x_0}^{x_3} \left| u_{x,x}(x,y_0) + \frac{1}{2}\xi_{,x}^2(x,y_0)\right| \dx \\
- \left| \fint_{x_0}^{x_1} u_x(x_r,0) - u_x(x_r,y_0) \ud x_r \right| -
\left| \fint_{x_2}^{x_3} u_x(x_l,0) - u_x(x_l,y_0) \ud x_l\right| \\
\ge \frac{1}{2}\int_{x_1}^{x_2} \xi_{,x}^2(x,0) \dx - 1/16 -
\int_{x_0}^{x_3} \left| u_{x,x}(x,y_0) + \frac{1}{2}\xi_{,x}^2(x,y_0)\right| \dx,
\end{multline}
where the last inequality follows from~\eqref{kappa3} (averaging an integral adds a factor
of reciprocal length to it, e.g.
$\frac{1}{512} \frac{1}{|x_0 - x_1|} \le \frac{1}{512} \frac{1}{1/16} = \frac{1}{32}$)
and~\eqref{bdry} (which is used to replace $u_x(x_l,0) - u_x(x_r,0)$ by $\int \xi_{,x}^2/2 \dx$).
We estimate the first term on the right-hand side by
\begin{equation} \nonumber
\frac{1}{2}\int_{x_1}^{x_2} \xi_{,x}^2(x,0) \dx =
\frac{1}{2} \int_{x_1}^{x_2} (2 \cos(2\pi x w_0^{-1}))^2 \dx \ge
2 \int_{-5/16}^{5/16} \cos^2(2\pi x w_0^{-1}) \dx \ge 1/4,
\end{equation}
where we used that $\int_I \cos^2(\theta t) \ud t \ge |I|/4$ provided that $\theta |I| > 2$.
(In the present setting $|I| = 10/16$ and $\theta = 2\pi/w_0 \geq 2\pi/\cw$ by~\eqref{wgeh}; since
we always assume $\cw \leq 1$, the condition holds with room to spare.) This concludes the proof
since either we have
\begin{equation} \nonumber
\int_{x_0}^{x_3} \left| u_{x,x}(x,y_0) + \frac{1}{2}\xi_{,x}^2(x,y_0)\right| \dx \ge 1/16
\end{equation}
or else
\begin{equation} \nonumber
\frac{1}{2}\int_{x_0}^{x_3} \xi_{,x}^2(x,y_0) \dx \ge
\frac{1}{4} - \frac{1}{16} -
\int_{x_0}^{x_3} \left| u_{x,x}(x,y_0) + \frac{1}{2}\xi_{,x}^2(x,y_0)\right| \dx \ge \frac{1}{8}.
\end{equation}

\end{proof}

\subsection{The energy required for coarsening}

When the type I deformation achieves the optimal scaling, we expect that the wrinkles must
coarsen as one moves down the sheet. The following lemma justifies this, by giving a
lower bound for the maximal amplitude of $\xi$.

\begin{lemma}\label{lmK}
There exists a universal constant $C_1 > 0$ such that
\begin{equation} \nonumber
K := \max_{y \in (-\l,0)} \int_{x_0}^{x_3} \xi^2(x,y) \dx \ge
C_1 \left( \frac{h^2 \l}{\delta} \wedge 1\right).
\end{equation}
\end{lemma}

\begin{proof}
If $\delta \ge \l/32^2$ the assertion is almost trivial; indeed, since
$x_3 - x_0 > 1/2$ and $h \le w_0 \le \cw$ (see~\eqref{wgeh}), we have
\begin{equation} \nonumber
K \ge \int_{x_0}^{x_3} \xi^2(x,0) \dx \gtrsim w_0^2 \ge h^2 \ge \frac{1}{(32)^2} \frac{h^2 \l}{\delta}.
\end{equation}
in this case.

If on the other hand $\delta \le \frac{\l}{32^2}$ then we argue as follows.
For any $y \in (-\l,0)$ Lemma~\ref{dichotomy3} implies
\begin{equation}\label{eq1}
\left(\int_{x_0}^{x_3} \xi_{,x}^2(x,y) \dx\right)^2 +
\int_{x_0}^{x_3} \left|u_{x,x}(x,y) + \xi_{,x}^2(x,y)/2\right| \dx \ge \frac{1}{16}.
\end{equation}
Using the interpolation inequality~\eqref{interpolation} and $|I|^2 = |x_3-x_0|^2 \ge 3/4$ we obtain
\begin{multline} \nonumber
\left( \int_{x_0}^{x_3} |\xi_{,x}(x,y)|^2 \dx \right)^2 \\
\le C_{int}^2\left[
\left( \int_{x_0}^{x_3} |\xi(x,y)|^2 \dx \right)^{1/2}
\left( \int_{x_0}^{x_3} |\xi_{,xx}(x,y)|^2 \dx \right)^{1/2} +
\frac{3}{4} \int_{x_0}^{x_3} |\xi(x,y)|^2 \dx \right]^2 \\
\le \bar C \left[ K \int_{x_0}^{x_3} |\xi_{,xx}(x,y)|^2 \dx + K^2\right].
\end{multline}
We integrate~\eqref{eq1} in $y$ and use the previous inequality to obtain
\begin{align*} \nonumber
\l/16 &\le \int_{-\l}^0 \left(\int_{x_0}^{x_3} \xi_{,x}^2(x,y) \dx\right)^2 \dy +
\int_{-\l}^0  \int_{x_0}^{x_3} \left|u_{x,x}(x,y_0) + \xi_{,x}^2(x,y_0)/2\right| \dxy \\  \nonumber
&\le \bar C K \int_{-\l}^0  \int_{x_0}^{x_3} |\xi_{,xx}(x,y)|^2 \dx \dy + \bar C \l K^2 \\
&\quad + \l^{1/2}
\left( \int_{-\l}^0 \int_{x_0}^{x_3} \left|u_{x,x}(x,y_0) + \xi_{,x}^2(x,y_0)/2\right|^2 \dxy \right)^{1/2} \\
\nonumber &\overset{\eqref{otherterms}}{\le} \bar CK \delta h^{-2} + \bar C \l K^2 + \delta^{1/2} \l^{1/2} \le
\bar CK \delta h^{-2} + \bar C \l K^2 + \l/32.
\end{align*}
using the assumption that $\delta \le \l/(32)^2$. We have obtained
\begin{equation} \nonumber
\l/32 \lesssim K \delta h^{-2} + l_0K^2,
\end{equation}
which implies
\begin{equation} \nonumber
K = \max_{y \in (-\l,0)} \int_{x_0}^{x_3} |\xi(x,y)|^2 \dx \ge
C_1 \left( \frac{h^2 \l}{\delta} \wedge 1\right).
\end{equation}
\end{proof}

The next lemma shows that coarsening costs energy:

\begin{lemma}\label{mainlemma}
Let
\begin{align} \nonumber
a := \int_{-1/2}^{1/2} \xi^2(x,0) \dx,\qquad b := \max_{y \in (-\l,0)} \int_{-1/2}^{1/2} \xi^2(x,y) \dx,
\end{align}
and assume
\begin{equation}\label{ah}
b \ge 4a, \qquad h^2 \le Da
\end{equation}
for some constant $D > 0$. Then there exists a constant $C_2(D) > 0$, depending on $D$, such that
\begin{multline}\label{eq5}
 \frac{\tau L}{8} \int_{-\l}^0 \int_{-1/2}^{1/2} \xi_{,y}^2(x,y) \dx \dy +
 h^2 \int_{-\l}^0 \int_{-1/2}^{1/2} \xi_{,xx}^2(x,y) \dx \dy \\
 + \int_{-\l}^0 \int_{-1/2}^{1/2} \left|u_{x,x}(x,y) + \xi_{,x}^2(x,y)/2\right|^2 \dx \dy \ge
 C_2(D) h \sqrt{\tau L} \ln(\bar b/a),
\end{multline}
where $\bar b := b \wedge (8C_{int})^{-1}$.
\end{lemma}

\begin{proof}
Let
\begin{equation} \nonumber
N := \lfloor\log_2(\bar b/a)/2\rfloor \ge 1,
\end{equation}
where $N \ge 1$ follows from \eqref{ah} and the fact that
$4a=2w_0^2\pi^{-2} \le 2\cw^2\pi^{-2} \le (8C_{int})^{-1}$ (the last inequality is a
smallness condition on $\cw$).
Then $\bar b \ge 2^{2N} a$ and for $i = 0,\ldots,N$ we can define
\begin{equation} \nonumber
 y_i := \max \left\{y \in [-\l,0] : \int_{-1/2}^{1/2} \xi^2(x,y) \dx = 2^{2i} a\right\}
\end{equation}
(note that $y_0 = 0$).

We will prove that for $i=0,\ldots,N-1$ the left-hand side of~\eqref{eq5} in
$(y_{i+1},y_i) \times (-1/2,1/2)$ is of order at least $h\sqrt{\tau L}$.
Fix $i \in \{ 0, \ldots, N-1 \}$. We see
\begin{equation}\label{est5}
\begin{aligned}
2^{2i+2}a &= \int_{-1/2}^{1/2} \xi^2(x,y_{i+1}) \dx =
\int_{-1/2}^{1/2} \left| \left( \int_{y_{i+1}}^{y_i} \xi_{,y}(x,y) \dy \right) + \xi(x,y_i)\right|^2 \dx \\
&\le 2\int_{-1/2}^{1/2} \left| \int_{y_{i+1}}^{y_i} \xi_{,y}(x,y) \dy \right|^2 \dx +
2\int_{-1/2}^{1/2} \xi^2(x,y_i) \dx \\
&\le 2(y_i - y_{i+1}) \int_{-1/2}^{1/2} \int_{y_{i+1}}^{y_i} \xi_{,y}^2(x,y) \dy \dx + 2 \cdot 2^{2i}a,
\end{aligned}
\end{equation}
which immediately implies
\begin{equation}\label{eq4}
(y_i - y_{i+1}) \int_{-1/2}^{1/2} \int_{y_{i+1}}^{y_i} \xi_{,y}^2(x,y) \dy \dx \ge 2^{2i} a.
\end{equation}

By Lemma~\ref{dichotomy3}, for any $y \in (-\l,0)$ either
\begin{equation} \nonumber
 \int_{-1/2}^{1/2} \left|u_{x,x}(x,y) + \xi_{,x}^2(x,y)/2\right|^2 \dx \ge 1/256
\end{equation}
or else
\begin{equation} \nonumber
 \int_{-1/2}^{1/2} \xi_{,x}^2(x,y) \dx \ge 1/4.
\end{equation}
In the latter case, for $y \in (y_{i+1},y_i)$ it follows from the interpolation
inequality~\eqref{interpolation} that
\begin{align*}
1/16 &\le \left( \int_{-1/2}^{1/2} \xi_{,x}^2(x,y) \dx \right)^2 \\
&\le C_{int}^2 \left[
\left( \int_{-1/2}^{1/2} \xi^2(x,y) \dx \right)^\oh \left( \int_{-1/2}^{1/2} \xi_{,xx}^2(x,y) \dx \right)^\oh +
\int_{-1/2}^{1/2} \xi^2(x,y) \dx
\right]^2 \\
&\le 2C_{int}^2 \left( \int_{-1/2}^{1/2} \xi^2(x,y) \dx \right)
\left( \int_{-1/2}^{1/2} \xi_{,xx}^2(x,y) \dx \right) + 2C_{int}^2 \bar b^2 \\
&\le 2C_{int}^22^{2(i+1)}a \int_{-1/2}^{1/2} \xi_{,xx}^2(x,y) \dx + \frac{2C^2_{int}}{64 C_{int}^2},
\end{align*}
and we can absorb the last term into the left-hand side to get
\begin{equation}\nonumber
2^{2(i+1)}a \int_{-1/2}^{1/2} \xi_{,xx}^2(x,y) \dx \gtrsim 1.
\end{equation}

Using~\eqref{ah} we obtain
\begin{multline} \nonumber
\int_{-1/2}^{1/2} \left|u_{x,x}(x,y) + \xi_{,x}^2(x,y)/2\right|^2 \dx +
h^2 \int_{-1/2}^{1/2} \xi_{,xx}^2(x,y) \dx \\
\gtrsim \min\left( 1/256, h^2/\left(a2^{2(i+1)}\right)\right) \gtrsim C(D)\frac{h^2}{2^{2i}a}.
\end{multline}
We integrate this inequality in $y$ over $(y_{i+1},y_i)$ and use
\eqref{eq4} to get
\begin{multline} \nonumber
\frac{\tau L}{8} \int_{y_{i+1}}^{y_i} \int_{-1/2}^{1/2} \xi_{,y}^2(x,y) \dx \dy +
h^2 \int_{y_{i+1}}^{y_i} \int_{-1/2}^{1/2} \xi_{,xx}^2(x,y) \dx \dy \\
+ \int_{y_{i+1}}^{y_i} \int_{-1/2}^{1/2}  \left|u_{x,x}(x,y) + \xi_{,x}^2(x,y)/2\right|^2 \dx \dy \ge
\frac{\tau L}{8} \frac{2^{2i}a}{y_i - y_{i+1}} + C(D)\frac{(y_i - y_{i+1}) h^2}{2^{2i}a} \\
\gtrsim C(D)h \sqrt{\tau L},
\end{multline}
where the last inequality follows from the AM-GM inequality. Since $N \gtrsim \log(\bar b/a)$,
summing the previous relation for $i=0,\ldots,N-1$ implies~\eqref{eq5}.
\end{proof}

\subsection{The lower bound half of Theorem~\ref{thm1}} 
We are finally ready to prove the lower bound half of Theorem \ref{thm1}.
We assume throughout the following discussion that the
assumptions of Theorem~\ref{thm1} are valid. As explained at the beginning of this section, we shall
argue by contradiction; in particular, shall assume that $(u,\xi)$ satisfies~\eqref{deltaless}.

By Corollary~\ref{dich} there exist $-1/2 < x_0 < -7/16$, $-3/8 < x_1 < -5/16$,
$5/16 < x_2 < 3/8$, $7/16 < x_3 < 1/2$, such that either there exist $l_0 \in (0,L/2)$ with the property
that for any $y \in (-l_0,0)$:
\begin{equation}\label{lbpf1}
\begin{aligned}
\left| \int_{x_0}^{x_1} u_x(x,y) - u_x(x,0) \dx \right| \le 1/512,\\
\left| \int_{x_2}^{x_3} u_x(x,y) - u_x(x,0) \dx \right| \le 1/512,
\end{aligned}
\end{equation}
and
\begin{equation}\label{lb-ing1}
\delta \max (\l, \l^3) \ge C_0,
\end{equation}
or else~\eqref{lbpf1} holds for all $y \in (-L/2,0)$ (in this case $l_0 = L/2$).

We use the notation from Lemma~\ref{mainlemma}, i.e.
\begin{equation} \nonumber
a = \int_{-1/2}^{1/2} \xi^2(x,0) \dx,\qquad
\b = \left( \max_{y \in (-\l,0)} \int_{-1/2}^{1/2} \xi^2(x,y) \dx \right) \wedge (8C_{int})^{-1}.
\end{equation}
We distinguish between two cases: $\b \ge 4a$ (when the wrinkles are expected to coarsen) and $\b \le 4a$
(when they are not).
\medskip

\noindent {\bf Case 1: $\b \ge 4a$.} Using the same idea as
in~\eqref{est5} we get that
\begin{equation} \nonumber
\b \le 2  l_0 \iint_{\Omega} \xi_{,y}^2 \dxy + 2a.
\end{equation}
Using $\b \ge 4a$ and $a \sim w_0^2$, the previous relation implies $\iint \xi^2_{,y} \gtrsim w_0^2 l_0^{-1}$.
Now it follows from~\eqref{xiy} that
\begin{equation}\label{speciallb}
 \delta \gtrsim (\tau L) \iint \xi^2_{,y} \gtrsim w_0^2 \tau L l_0^{-1}.
\end{equation}


Next we want to show that
\begin{equation}\label{lbtoshow}
\delta \gtrsim h\sqrt{\tau L} \ln \left( w_0^{-2} \left( \frac{h l_0}{\sqrt{\tau L}} \wedge 1 \right) + 1\right) . \end{equation}
We distinguish two subcases, when $hl_0 / (w_0^2 \sqrt{\tau L}) \le 1$ and when $hl_0 / (w_0^2 \sqrt{\tau L}) > 1$.
The definition of the first subcase is equivalent to $h\sqrt{\tau L} \le w_0^2 \tau L l_0^{-1}$, which together
with~\eqref{speciallb} implies
\begin{equation} \nonumber
\delta \gtrsim w_0^2 \tau L l_0^{-1} \ge h\sqrt{\tau L} \gtrsim
h\sqrt{\tau L} \ln \left( \frac{h l_0}{w_0^2 \sqrt{\tau L}} + 1\right) \ge
h\sqrt{\tau L} \ln \left( w_0^{-2} \left( \frac{h l_0}{\sqrt{\tau L}} \wedge 1\right) + 1\right),
\end{equation}
where the second-last inequality follows from the fact that the argument in the logarithm is less than $2$.

Turning to the second subcase: we must show~\eqref{lbtoshow} when $hl_0 / (w_0^2 \sqrt{\tau L}) > 1$.
Let $C_1$ be the constant that was chosen in Lemma~\ref{lmK}.
We may suppose that $\delta$ is small in the sense that
\begin{equation}\label{lbtoshow1}
\delta \le 2\C h\sqrt{\tau L} \ln \left( \frac{h l_0}{w_0^2 \sqrt{\tau L}} + 1\right),
\end{equation}
since the opposite inequality implies~\eqref{lbtoshow}. By the initial conditions~\eqref{bdry}
\begin{equation} \nonumber
 a = \int_{-1/2}^{1/2} \xi^2(x,0) \dx = w_0^2 \pi^{-2}/2,
\end{equation}
in particular thanks to $w_0 \ge h$ (see~\eqref{wgeh}), condition \eqref{ah} holds with
$D = 2\pi^2$. By Lemma \ref{lm-xi1} $2\delta$ is more than the left-hand side in~\eqref{eq5},
and so using Lemma~\ref{mainlemma} and  Lemma~\ref{lmK} we get that
\begin{equation}\label{estdelta1}
2\delta \ge C_2 h\sqrt{\tau L} \ln (\b/a) \ge
C_2 h\sqrt{\tau L} \ln \left( \frac{2 C_1 \pi^2}{w_0^2} \left(  \frac{h^2 l_0}{\delta} \wedge 1\right) \right).
\end{equation}
We now observe using~\eqref{lbtoshow1} that
\begin{align*}
\frac{2C_1 \pi^2 h^2 l_0}{w_0^2 \delta} &= 
\frac{h l_0}{w_0^2 \sqrt{\tau L}} \cdot \frac{2C_1 \pi^2 h \sqrt{\tau L}}{\delta} \\
& \geq \frac{h l_0}{w_0^2 \sqrt{\tau L}} \cdot \frac{1}{\log \left( \frac{hl_0}{w_0^2 \sqrt{\tau L}} + 1 \right)} ,
\end{align*}
so \begin{equation}\label{estdelta2}
C_2 h\sqrt{\tau L} \ln \left( 2 C_1 \pi^2 \frac{h^2 l_0}{w_0^2 \delta}\right) \ge
C_2 h\sqrt{\tau L} \left(
\ln \left( \frac{h l_0}{w_0^2 \sqrt{\tau L}} \right) -
\ln \ln \left( \frac{h l_0}{w_0^2 \sqrt{\tau L}} + 1 \right) \right).
\end{equation}
Since $\ln(t) - \ln \ln (t+1) \ge \ln(t+1)/2$ for $t > 1$, combination of~\eqref{estdelta1}
and~\eqref{estdelta2} implies
\begin{equation}\label{lb-ing3}
\delta \ge \frac{C_2}{4} h\sqrt{\tau L}
\min\left( \ln \left( \frac{h l_0}{w_0^2 \sqrt{\tau L}} + 1 \right), 4\ln\left( \frac{2C_1 \pi^2}{w_0^2} \right)\right)
\gtrsim h\sqrt{\tau L} \ln \left( w_0^{-2} \left( \frac{h l_0}{\sqrt{\tau L}}\wedge 1\right) + 1\right),
\end{equation}
where we used that $hl_0/(w_0^2 \sqrt{\tau L}) > 1$ and $w_0 \leq \cw$ (with $\cw$ sufficiently small).
This completes the proof of~\eqref{lbtoshow} in the second subcase.

We now easily conclude the validity of the lower bound half of Theorem~\ref{thm1} when $\b \ge 4a$.
In fact, if $l_0 < L/2$
then adding \eqref{lbtoshow}, \eqref{speciallb}, and \eqref{lb-ing1} gives
\begin{equation} \nonumber
\delta \gtrsim h\sqrt{\tau L} \ln \left( w_0^{-2} \left( \frac{h l_0}{\sqrt{\tau L}}\wedge 1\right) + 1 \right) +
w_0^2 \tau L l_0^{-1} + \min \left\{ l_0^{-1}, l_0^{-3} \right\} ,
\end{equation}
while if $l_0 = L/2$ we have
\begin{equation} \nonumber
\delta \gtrsim h\sqrt{\tau L} \ln \left( w_0^{-2} \left( \frac{h L}{\sqrt{\tau L}}\wedge 1\right) + 1 \right).
\end{equation}
\medskip

\noindent {\bf Case 2: $\b \le 4a$.}
Since $\b \le 4a$ and $4a = 2w_0^2 \pi^{-2} \le 2\cw^2 \pi^{-2} < (8C_{int})^{-1}$
for small enough $\cw$, we immediately see that the value of
$b = \max_{y \in (-\l,0)} \int_{-1/2}^{1/2} \xi^2(x,y) \dx$ agrees with $\b$.

To prove the lower bound in this case it is sufficient to show that
\begin{equation}\label{speciallb2}
\delta \gtrsim \min \left( h^2 w_0^{-2}L, \min_{l \in (0,L)} \left( h^2 w_0^{-2} l + w_0^2 \tau L l^{-1} +
\min\left(l^{-1},l^{-3}\right)\right) \right).
\end{equation}
Indeed, since $t \ge \ln(1+t)$ we have
\begin{gather} \nonumber
h^2 w_0^{-2}l \ge h\sqrt{\tau L}\ln\left( \frac{h l}{w_0^2 \sqrt{\tau L}} + 1\right) \ge
h\sqrt{\tau L}\ln\left( w_0^{-2} \left( \frac{h l}{\sqrt{\tau L}} \wedge 1\right) + 1\right)
\end{gather}
for any $l > 0$, and so the desired lower bound for $\delta$ follows from~\eqref{speciallb2}.

The rest of this section is devoted to proving~\eqref{speciallb2}. We note for future reference that
the prescribed boundary conditions~\eqref{bdry} satisfy
\begin{align*}
||\xi(\cdot,0)||_{L^2(-1/2,1/2)}^2 &= w_0^2\pi^{-2}/2 = a,
\\ ||\xi_{,x}(\cdot,0)||_{L^2(-1/2,1/2)}^2 &= 2,
\\ ||\xi_{,xx}(\cdot,0)||_{L^2(-1/2,1/2)}^2 &= 8 \pi^2  w_0^{-2}.
\end{align*}

If $||\xi_{,xx}(\cdot,y)||_{L^2(-1/2,1/2)}^2 \ge \pi^2 w_0^{-2}/(8C^2_{int})$ for all $y \in (-L/2,0)$
then our task is almost trivial, since
\begin{equation}\label{lb-special-p1}
 \delta \overset{\eqref{otherterms}}{\ge} h^2 \iint_\Omega \xi_{,xx}^2(x,y) \dxy \gtrsim h^2 w_0^{-2} L,
\end{equation}
which implies~\eqref{speciallb2}.

Thus it suffices to consider the situation when
$||\xi_{,xx}(\cdot,y)||_{L^2(-1/2,1/2)}^2 < \pi^2 w_0^{-2}/(8C^2_{int})$
for some $y \in (-L/2,0)$. Recalling (see Remark \ref{smoothness}) that we may assume $(u,\xi)$ are smooth,
consider the first $y_0 \in (-L/2,0)$ (i.e. the one with smallest $|y_0|$) such that
\begin{equation} \nonumber
\int_{-1/2}^{1/2} \xi_{,xx}^2(x,y_0) \dx \le \pi^{2} w_0^{-2}/(8C^2_{int}).
\end{equation}
Note that by the definition of $y_0$ we have $\int \xi_{,xx}^2(x,y) \dx \ge \pi^2 w_0^{-2} / (8C_{int}^2), $
for all $y \in (y_0,0)$, hence
\begin{gather} \label{est-later}
\delta \overset{\eqref{otherterms}}{\ge} h^2 \int_{y_0}^0 \int_{-1/2}^{1/2} \xi_{,xx}^2(x,y) \dx \dy \gtrsim
h^2 w_0^{-2}|y_0|.
\end{gather}
Let
\begin{equation} \nonumber
\eta(x) := \xi(x,0) - \xi(x,y_0)\qquad \textrm{for } x \in (-1/2,1/2)
\end{equation}
and observe that
\begin{multline}\label{etaL2}
||\eta||_{L^2}^2 = \int_{-1/2}^{1/2} \eta^2 \dx =
\int_{-1/2}^{1/2} |\xi(x,0) - \xi(x,y_0)|^2 \dx\\
\le \int_{-1/2}^{1/2} \left( |y_0| \int_{y_0}^0 \xi_{,y}^2(x,y) \dy \right) \dx \le
|y_0| \iint_{\Omega} \xi_{,y}^2 \dxy .
\end{multline}

We now distinguish between two subcases: when $|y_0| > \l$ and when $|y_0| \le \l$.
Focusing initially on the subcase $|y_0| > \l$: our main task is to show that \begin{equation} \label{w0tauLa}
\delta \gtrsim w_0^2 \tau L |y_0|^{-1}.
\end{equation}
Since $y_0 \not \in (-\l,0)$, the relation $b \le 4a$ does not
imply $\int_{-1/2}^{1/2} \xi^2(x,y_0) \dx \le 4a$. Therefore we must consider both of the following
possibilities:
\begin{enumerate}
\item[(i)] Suppose $\int_{-1/2}^{1/2} \xi^2(x,y_0) \dx \ge 4a$. Then the triangle inequality gives
$||\eta||_{L^2} \ge 2\sqrt{a} - \sqrt{a} = \sqrt{a} = w_0/(\sqrt{2}\pi)$, and~\eqref{etaL2} gives
$$
w_0^2/(2\pi^2) \le ||\eta||_{L^2}^2 \le |y_0| \iint_\Omega \xi_{,y}^2 \dxy \le |y_0| 8 \delta / (\tau L),\\
$$
which implies~\eqref{w0tauLa}.

\item[(ii)] Suppose $\int_{-1/2}^{1/2} \xi^2(x,y_0) \dx \le 4a$. Then we see using the interpolation
inequality~\eqref{interpolation}
\begin{align*}
||\xi_{,x}(\cdot,y_0)||_{L^2(-1/2,1/2)}^4 &= \left( \int_{-1/2}^{1/2} \xi_{,x}^2(x,y_0) \dx \right)^2
\\ &\le 2C^2_{int}\int_{-1/2}^{1/2} \xi^2(x,y_0) \dx \int_{-1/2}^{1/2} \xi_{,xx}^2(x,y_0) \dx
\\ &\quad + 2C_{int}^2 \left( \int_{-1/2}^{1/2} \xi^2(x,y_0) \dx \right)^2
\\ &\le 2C^2_{int}\; 4a\; \pi^2 w_0^{-2} / (8C^2_{int}) + 32C_{int}^2 a^2
\\ &= 1 / 2 +  32C_{int}^2a^2 \le 1,
\end{align*}
where the last inequality holds for small enough $a$ (hence small enough $\cw$).
Therefore
\begin{align} \nonumber
||\eta_{,x}||_{L^2} &\ge ||\xi_{,x}(\cdot,0)||_{L^2} - ||\xi_{,x}(\cdot,y_0)||_{L^2} \ge
\sqrt{2} - 1,\\ \nonumber
||\eta_{,xx}||_{L^2} &\le ||\xi_{,xx}(\cdot,0)||_{L^2} - ||\xi_{,xx}(\cdot,y_0)||_{L^2} \lesssim w_0^{-1}.
\end{align}
By interpolation~\eqref{interpolation}
\begin{equation} \nonumber
(\sqrt{2}-1)^2 \le ||\eta_{,x}||^2_{L^2} \le
C_{int} ||\eta||_{L^2} ||\eta_{,xx}||_{L^2} + C_{int}||\eta||_{L^2}^2.
\end{equation}
The second term satisfies
$$
C_{int}\|\eta\|_{L^2}^2 \le C_{int} ( \|\xi(\cdot,0)\|_{L^2} + \|\xi(\cdot,y_0)\|_{L^2} )^2 \le
16C_{int}^2a \le 8C_{int} \cw^2 \pi^{-2} \le (\sqrt{2}-1)^2/2
$$
for small enough $\cw$. Thus it can be absorbed into the left-hand side of the previous relation. Combining these
results with~\eqref{etaL2}, we get
\begin{equation} \nonumber
1 \lesssim |y_0| \left(\iint_{\Omega} \xi_{,y}^2 \dxy \right) w_0^{-2},
\end{equation}
and so
$$
\delta \ge \frac{\tau L}{8} \iint_\Omega \xi_{,y}^2 \gtrsim w_0^2 \tau L  |y_0|^{-1},
$$
confirming the validity of~\eqref{w0tauLa}.
\end{enumerate}

To conclude the treatment of the subcase $|y_0| > \l$ we must show that~\eqref{speciallb2} holds.
Combining~\eqref{est-later} with~\eqref{w0tauLa} gives
\begin{equation}\label{est-later2}
 \delta \gtrsim h^2 w_0^{-2} |y_0| + w_0^2 \tau L |y_0|^{-1}.
\end{equation}
If $\l = L/2$ then~\eqref{speciallb2} follows from~\eqref{est-later2} (using that $|y_0| > \l$). If on the other
hand $\l<L/2$ then we have (using again that $|y_0| > \l$)
\begin{multline} \nonumber
\delta \gtrsim h^2 w_0^{-2} |y_0| + w_0^2 \tau L |y_0|^{-1} + \min(l_0^{-1},l_0^{-3}) \\
\ge h^2 w_0^{-2} |y_0| + w_0^2 \tau L |y_0|^{-1} + \min(|y_0|^{-1},|y_0|^{-3}),
\end{multline}
which implies~\eqref{speciallb2}. The subcase $|y_0| > \l$ is now complete.
\medskip

Turning now to the other subcase, when $|y_0| \le \l$: using Lemma~\ref{dichotomy3} we
see that for any $y \in (-l_0,y_0)$ either
\begin{equation} \nonumber
1/4 \le \int_{-1/2}^{1/2} \xi_{,x}^2(x,y) \dx
\end{equation}
or
\begin{equation}\label{eq---}
 1/256 \le \int_{-1/2}^{1/2} |u_{x,x}(x,y) + \xi_{,x}^2(x,y)/2|^2 \dx.
\end{equation}
If for a given $y$ the first case is true, we use the interpolation inequality~\eqref{interpolation} to get
\begin{align*}
1/4 &\le \int_{-1/2}^{1/2} \xi_{,x}^2(x,y) \dx
\\ &\le C_{int} \left( \int_{-1/2}^{1/2} \xi^2(x,y) \dx \right)^\oh
\left( \int_{-1/2}^{1/2} \xi_{,xx}^2(x,y) \dx\right)^\oh
+ C_{int} \int_{-1/2}^{1/2} \xi^2(x,y) \dx.
\end{align*}
Since $b \le 4a \sim w_0^2$ and $C_{int} 4a \le 1/8$, the previous estimate implies
\begin{equation}\nonumber
1 \lesssim w_0^2 \int_{-1/2}^{1/2} \xi_{,xx}^2(x,y) \dx.
\end{equation}
We combine this estimate with~\eqref{eq---} to get
\begin{equation} \nonumber
\int_{-1/2}^{1/2} |u_{x,x}(x,y) + \xi_{,x}^2(x,y)/2|^2 + h^2 \xi_{,xx}^2(x,y) \dx \ge
\min(1/256, Ch^2w_0^{-2}) \gtrsim h^2w_0^{-2},
\end{equation}
where we used that $w_0 \ge h$ (see~\eqref{wgeh}). We integrate this in $y$ over $(-l_0,y_0)$ to get
\begin{equation} \nonumber
\int_{-l_0}^{y_0} \int_{-1/2}^{1/2} |u_{x,x}(x,y) + \xi_{,x}^2(x,y)/2|^2 + h^2 \xi_{,xx}^2(x,y) \dx \dy
\gtrsim h^2w_0^{-2} (l_0 - |y_0|).
\end{equation}
This combined with~\eqref{est-later} implies
\begin{equation}\label{lb-now1}
\delta \gtrsim h^2 w_0^{-2} l_0. 
\end{equation}
Since $b \le 4a$, we have that $\int_{-1/2}^{1/2} \xi^2(x,y_0) \dx \le 4a$ and arguing as
we did earlier (in case (ii)) gives~\eqref{w0tauLa}.

To conclude the treatment of the subcase $|y_0| \leq \l$ we must show that~\eqref{speciallb2} holds.
If $\l = L/2$ then~\eqref{lb-now1} gives
\begin{equation}
\delta \gtrsim h^2 w_0^{-2} \l = h^2 w_0^{-2} L/2,
\end{equation}
which implies~\eqref{speciallb2}. If on the other hand $\l < L/2$, then we know
$$
\delta \gtrsim h^2 w_0^{-2} l_0 + w_0^2 \tau L |y_0|^{-1} + \min(l_0^{-1},l_0^{-3}).
$$
Since $|y_0| \leq l_0$, we have that $w_0^2\tau L |y_0|^{-1} \ge w_0^2\tau L l_0^{-1}$, this gives
$$
\delta \gtrsim h^2 w_0^{-2} l_0 + w_0^2 \tau L \l^{-1} + \min(l_0^{-1},l_0^{-3}),
$$
which implies~\eqref{speciallb2}. This completes the treatment of the subcase $|y_0| \leq \l$. The lower bound
half of Theorem \ref{thm1} has now been fully established.



\providecommand{\bysame}{\leavevmode\hbox to3em{\hrulefill}\thinspace}
\providecommand{\MR}{\relax\ifhmode\unskip\space\fi MR }
\providecommand{\MRhref}[2]{%
  \href{http://www.ams.org/mathscinet-getitem?mr=#1}{#2}
}
\providecommand{\href}[2]{#2}

\end{document}